\pdfoutput=1
\documentclass{amsart}
\usepackage{amsmath,amssymb,amsthm}
\usepackage{verbatim}
\usepackage{graphicx}

\makeatletter
\newtheorem*{rep@theorem}{\rep@title}
\newcommand{\newreptheorem}[2]{
\newenvironment{rep#1}[1]{
 \def\rep@title{#2 \ref{##1}}%
 \begin{rep@theorem}}%
 {\end{rep@theorem}}}
\makeatother
\newcommand{\figwidth}{2in}

\newtheorem{thm}{Theorem}[section]
\newreptheorem{thm}{Theorem}
\newtheorem{lem}[thm]{Lemma}
\newtheorem{prop}[thm]{Proposition}

\newtheorem{definition}{Definition}[section]
\newcommand{\R}{\mathbb{R}}
\newcommand{\kO}{\mathcal{O}}
\newcommand{\Q}{\mathbb{Q}}
\newcommand{\uhr}{\upharpoonright}
\newcommand{\tP}{\tilde{P}}
\newcommand{\Nat}{\mathbb{N}}
\newcommand{\hband}{\hbox{ and }}

\begin{document}

\title{A lightface analysis of the differentiability rank}
\author{Linda Brown Westrick}
\date{June 29, 2013}
\email{westrick@math.berkeley.edu}

\begin{abstract}
  We examine the computable part of the differentiability hierarchy
  defined by Kechris and Woodin.  In that hierarchy, the rank of a
  differentiable function is an ordinal less than $\omega_1$ which
  measures how complex it is to verify differentiability for that
  function.  We show that for each recursive ordinal $\alpha>0$, the
  set of Turing indices of $C[0,1]$ functions that are differentiable
  with rank at most $\alpha$ is $\Pi_{2\alpha+1}$-complete. This
  result is expressed in the notation of Ash and Knight.
\end{abstract}

\maketitle

\section{Introduction}

The set of differentiable $C[0,1]$ functions is not Borel, but it can be represented hierarchically as an increasing union of Borel sets.  Three hierarchies for the differentiable functions have been proposed (see the work of Ki \cite{Ki97} for a summary).\nocite{ramsamujh}  Each hierarchy is defined using an \emph{ordinal rank}, a mapping from differentiable functions to countable ordinals, whose range is unbounded below $\omega_1$.  Here we focus on Kechris and Woodin's differentiability rank \cite{kw}, denoted $|\cdot |_{KW}$.  It decomposes the set $\mathcal{D}$ of differentiable $C[0,1]$ functions as $$\mathcal{D} = \bigcup_{\alpha<\omega_1} \{f : |f|_{KW} < \alpha\}$$ where each constituent of the union is Borel.  

Our contribution is a finer-grained, recursion-theoretic analysis of this hierarchy.  The lightface situation mirrors the boldface situation in many ways.  We begin with the observation (a corollary of the work in \cite{kw}) that the set $D$ of integer codes for computable differentiable $C[0,1]$ functions is a $\Pi_1^1$-complete set, and it decomposes as $$D = \bigcup_{\alpha<\omega_1^{CK}} \{c : c \text{ codes } f \text{ with } |f|_{KW}< \alpha\}$$ where each constituent of the union is hyperarithmetic.  Our results pinpoint the exact location of each constituent set in the transfinitely-extended Levy hierarchy.

\begin{repthm}{maintheorem}
For each nonzero $\alpha<\omega_1^{CK}$, the set $$\{c : c \text{ codes } f \text{ with } |f|_{KW}< \alpha+1\}$$ is $\Pi_{2\alpha+1}$-complete.
\end{repthm}
Here and throughout we use the notational convention of Ash and Knight \cite{ashknight} for a $\Sigma_\alpha$ set, discussed in Section \ref{sec2.2}.  We also analyze the limit case:

\begin{repthm}{limittheorem}
For each limit $\lambda <\omega_1^{CK}$, the set \mbox{$\{c : c \text{ codes } f \text{ with } |f|_{KW}< \lambda\}$} is $\Sigma_{\lambda}$-complete.
\end{repthm}

The study of differentiation through the lens of computable analysis
has typically involved restricting attention to the continuously
differentiable functions.  The definition of a computable function proposed by
Grzegorczyk and Lacombe, and further developed
by Pour-El and Richards and others (see \cite{grz}, \cite{lac}, \cite{pour-el}),
 has no notion of computability
for a discontinuous function.  Therefore, restricting differentiation
 to the continuously differentiable functions is a strategy for making
 questions such as ``Is differentiation computable?'' meaningful.
  The fact that $f\mapsto f'$ is not
computable was first demonstrated by Myhill \cite{Myhill}, who
constructed a computable function whose continuous derivative is not
computable.  

At the other end of the spectrum, computable functions that are not
everywhere differentiable have been studied.  Brattka, Miller and
 Nies (to appear) have used randomness notions to characterize the
 points at which all computable almost everywhere differentiable
 functions must be differentiable.
However, as far as the author is aware, the
everywhere differentiable functions with discontinuous derivatives have not yet
been studied in the setting of computable analysis.

Previously, Cenzer and Remmel \cite{cenzer} showed that $\{e : f_e
\text{ is continuously differentiable}\}$ is $\Pi^0_3$-complete, which
is the same as the $\alpha=1$ case of our Theorem \ref{maintheorem}.
They also showed that $\{e: f_e \text{ is continuously differentiable
and computable}\}$ is $\Sigma^0_3$-complete.  Again, only continuously
differentiable functions were considered.  By contrast, our aim is to
provide a clearer picture of the structure of the unrestricted set of
everywhere differentiable functions.

In Section \ref{preliminaries} we review the basic facts about computable $C[0,1]$ functions, the ordinals below $\omega_1^{CK}$, and $\Sigma_\alpha$ sets.  Then we introduce Kechris and Woodin's differentiability rank, and present what is known about $D$.  In Section \ref{sec2} we redefine the differentiability rank in a more
computationally convenient way, and use this definition to demonstrate
$\{c : c \text{ codes } f \text{ with } |f|_{KW}< \alpha+1\}$ is $\Pi_{2\alpha+1}.$  The meat of the paper is in Section \ref{sec3}, where
we address the question of completeness to prove both theorems above.

I would like to thank Theodore Slaman for many useful conversations
and a simpler proof of Lemma \ref{g0}, Ian Haken for carefully
reading a draft and providing valuable suggestions, and Rod Downey and an anonymous reviewer for
their good advice on presentation and digestibility.  However, any
errors or expository flaws are entirely the responsibility of the author.

\section{Preliminaries}\label{preliminaries}

This section provides background, essential definitions, methods
previously used to construct functions of different ranks, and
corollaries that are straightforward effectivizations of arguments in
the literature.  In section \ref{sec2.1} we establish some notation and review the basic facts about computable $C[0,1]$ functions.  In Section \ref{sec2.2} we introduce the recursive ordinals and use them to define $\Sigma_\alpha$-completeness.  In Section \ref{sec2.3} we define Kechris and Woodin's differentiability rank.  In Section \ref{buildingblocks} we familiarize the reader with the building blocks used in \cite{kw} to construct functions of arbitrary rank, as these essential elements are taken for granted in what follows.  In Section \ref{sec2.5} we establish more notation that is used throughout the paper.  Finally, in Section \ref{sec2.6} we present some necessary facts about computable differentiable functions that can be obtained by effectivizing existing work.

\subsection{Basic notions and encoding $C[0,1]$ functions}\label{sec2.1}

We use $\phi_e$ to denote to the $e$th Turing functional, and $W_e$
refers to the domain of $\phi_e$.  The jump of $X\in 2^\omega$ is
written $X'$, and the $n$th jump of $X$ is written $X^{(n)}$.  Turing reducibility is denoted by $\leq_T$ and one-reducibility by $\leq_1$.  We use $\langle n_1,\dots,n_k\rangle$ to denote a
single integer which represents the tuple $(n_1,\dots,n_k)$ according
to some standard computable encoding.  If $m$ is a number and $\sigma
= \langle n_1,\dots,n_k\rangle$ is a sequence, $m^\smallfrown \sigma$
and $\sigma^\smallfrown m$ denote $\langle m, n_1, \dots, n_k\rangle$
and $\langle n_1,\dots,n_k, m\rangle$ respectively.  If $T\subseteq
\mathbb{N}^{<\mathbb{N}}$ is a tree, let $T_n$ denote $\{\sigma :
n^\smallfrown\sigma \in T\}$, the $n$th subtree of $T$.  If $T$ is
well-founded, $|T|$ denotes its rank.

We identify the computable functions with the computable reals that
encode those functions.  Following \cite{kw}, all our functions are
real-valued with domain $[0,1]$.  For the encoding we use Simpson's
definition from \cite{simpson} because this encoding makes it
straightforward to determine the degree of unsolvability of various
statements.  For example, we will observe that ``$\phi_e$ encodes a
computable $C[0,1]$ function'' is $\Pi_2$.  However, the exact details
of the Simpson encoding are not needed beyond this section, and any of
the many equivalent definitions for a computable real-valued function
can be safely substituted.

In the following definition, $(a,r)\Phi(b,s)$ is shorthand for
$\exists n ((n,a,r,b,s,)\in \Phi)$, and $(a,r)<(a',r')$ means that
$|a-a'|+r' < r$.  The idea is that $(a,r)\Phi(b,s)$ should mean that
$f(B(a,r))\subseteq \overline{B(b,s)}$.

\begin{definition}\label{sdef}
  A code for a continuous partial function $f$ from $[0,1]$ to $\R$ is
  a set of quintuples $\Phi\subseteq \Nat \times \Q\cap[0,1] \times
  \Q^+ \times \Q \times \Q^+$ which satisfies:
\begin{enumerate}
\item if $(a,r)\Phi(b,s)$ and $(a,r)\Phi(b',s')$ then $|b-b'|\leq s+ s'$
\item if $(a,r)\Phi(b,s)$ and $(a',r')<(a,r),$ then $(a',r')\Phi(b,s)$
\item if $(a,r)\Phi(b,s)$ and $(b,s)<(b',s')$, then $(a,r)\Phi(b',s')$.
\end{enumerate}
\end{definition}

This set $\Phi$ is coded as a subset of $\mathbb{N}$ using the
standard encoding.  Some important facts can be seen from this
definition.  Firstly, it is $\Pi_2$ to check whether a given real
satisfies the above properties.  Secondly, the reals satisfying the
above might not represent total functions.  That is, for some points
$x$ in $[0,1]$ and some $\varepsilon$ there may not be an $a,r,b$ such
that $|x-a|<r$ and $(a,r)\Phi(b,\varepsilon)$.  However if the real
does represent a total function then, by the compactness of $[0,1]$,
for each $\varepsilon$ there is a finite set
$\{(a_i,r_i,b_i,s_i)\}_{i<p}$ such that the $(a_i,r_i)$ cover $[0,1]$
and for each $i$, $s_i\leq\varepsilon$ and $(a_i,r_i)\Phi(b_i,s_i)$.
Therefore, ``$\phi_e$ encodes a $C[0,1]$ function'' is a
$\Pi_2$ statement: $\phi_e$ is total, and the corresponding real satisfies Definition \ref{sdef}, and for all $\varepsilon$ there is a finite cover as described above.  Let $f_e$ denote the $C[0,1]$ function encoded by $\phi_e$.  Note that any function encoded using this convention
is, by necessity, continuous.

If $f$ is any computable $C[0,1]$ function and $z$ and $x$ any rational
numbers, the statement $f(x)>z$ is $\Sigma_1$, because $f(x)>z$ if and
only if there are $\delta$, $b$ and $\varepsilon$ such that
$(x,\delta)\Phi(b,\varepsilon)$ and $b-\varepsilon > z$.

We will also freely make use of the fact that addition,
multiplication, division, and composition of computable functions are
computable.  For details we refer the reader to \cite{simpson}.

\subsection{Kleene's $\kO$ and the notion of a $\Sigma_\alpha$-complete set}\label{sec2.2}

Kleene's $\kO$ is a way of
encoding the recursive ordinals as natural
numbers.  First one defines a relation $<_\kO$ on $\Nat$ as the least
relation closed under the following properties:
\begin{enumerate}
\item $1<_\kO 2.$
\item If $a<_\kO b$ then $b <_\kO 2^b$.
\item If $\phi_e(n)$ is total and $\phi_e(n)<_\kO\phi_e(n+1)$ for all $n$, then  $\phi_e(n)<_\kO3\cdot5^e$ for all $n$.
\item If $a<_\kO b$ and $b<_\kO c$ then $a<_\kO c.$ 
\end{enumerate}
The field of this relation is called Kleene's $\kO$.  One can show
that $\kO$ is a $\Pi_1^1$-complete set, that $<_\kO$ is well-founded,
and for each $a\in\kO$, the set $\{b : b<_\kO a\}$ is well ordered and
computably enumerable.  (See \cite{sacks} for details).  Therefore,
for each $a\in\kO$ there is a well-defined ordinal $|a|_\kO =
\hbox{ot}(\{b : b<_\kO a\})$.  In this situation $a$ is called an \emph{ordinal notation} for $|a|_\kO$.  If an ordinal has an ordinal
notation in $\kO$, it is called a \emph{constructive ordinal}.  Note
that there are infinitely many ordinal notations corresponding to each
constructive ordinal $\alpha \geq \omega$.  There are only countably
many constructive ordinals and these form an initial segment of the
ordinals.  The least nonconstructive ordinal is called
$\omega_1^{CK}$, ``the $\omega_1$ of Church and Kleene''.

We will use the fact that it is computable to add  ordinal
notations in a way that is consistent with their corresponding
ordinals.  

The constructive ordinals have an important equivalent
characterization.  They are exactly the ranks of the recursive
well-founded relations.  This will be used to establish that the
differentiability ranks of the computable functions are the
constructive ordinals.

We recall the Levy hierarchy for $n<\omega$.  A set $X$ is said to be $\Sigma_n$ (respectively $\Pi_n$) if $X\leq_1 \emptyset^{(n)}$ (respectively $\overline{\emptyset^{(n)}}$), and $X$ is $\Sigma_n$-complete if $X \equiv_1 \emptyset^{(n)}$ (and similarly for $\Pi_n$-completeness).

The ordinal notations provide a natural way to extend the notion of
the Turing jump, and thus the Levy hierarchy, through the ordinals
less than $\omega_1^{CK}$.  Define $H_1=\emptyset$, $H_{2^b} =
(H_b)'$, and $H_{3\cdot 5^e} =
\{\langle x, n \rangle : x \in H_{\phi_e(n)}\}$.  Spector
\cite{spector} showed that if $|a|_\kO = |b|_\kO$, then $H_a \equiv_T
H_b$.  Therefore, $H_{2^a}\equiv_1 H_{2^b}$, and thus there is a
well-defined notion of one-reducibility and completeness at the
successor levels.  We define the notions of $\Sigma_\alpha$ and $\Pi_\alpha$ for infinite ordinals following \cite{ashknight}:
\begin{definition}
  Let $\alpha <\omega_1^{CK}$ be an infinite ordinal and let $X \in
  2^\omega$.  Then $X$ is $\Sigma_\alpha$ if $X \leq_1 H_{2^a}$ for any
  $a$ such that $|a|_\kO=\alpha$, and $X$ is $\Sigma_\alpha$-complete
  if $X\equiv_1 H_{2^a}$ for any such $a$.  The $\Pi_\alpha$ and
  $\Pi_\alpha$-complete sets are defined similarly.
\end{definition}
Note that using this definition, $(\emptyset^{(\omega)})'$ is a
$\Sigma_\omega$-complete set.  There is a conflicting notational
convention, found in \cite[pg. 259]{soare}, in which
$(\emptyset^{(\omega)})'$ is classified $\Sigma_{\omega+1}$-complete,
and the symbol $\Sigma_\omega$ is not defined.  We prefer the notation
of \cite{ashknight} because it is more consonant with the definition
of the rank function.  As will be seen, to determine whether the core
rank-ascertaining process terminates at a limit stage, it is necessary
to use a quantification over the results of the previous stages, not
merely a unified presentation of them.

We fix a particular (but arbitrary) path $\mathcal{P}$ through $\kO$
and define $\emptyset^{(\alpha)}$ for each $\alpha < \omega_1^{CK}$ by
$\emptyset^{(\alpha)} = H_a$, where $a$ is the unique $a\in\mathcal{P}$ such
that $|a|_\kO = \alpha$.  (We call $\mathcal{P}$ a path through $\kO$
if $\mathcal{P}\subseteq \kO$ is $<_\kO$-linearly ordered and contains
an ordinal notation for each $\alpha<\omega_1^{CK}$.)

Because $\emptyset^{(\alpha+1)}$ is the canonical
$\Sigma_\alpha$-complete set when $\alpha>\omega$, we
follow \cite{noamconvention} in defining $$\emptyset_{(\alpha)}
= \begin{cases} \emptyset^{(\alpha)} & \text{ if
} \alpha<\omega \\ \emptyset^{(\alpha+1)} & \text{ if
} \alpha \geq \omega \end{cases}$$ so that $\emptyset_{(\alpha)}$ is
always the canonical $\Sigma_\alpha$-complete set.  In addition, we
identify $\alpha$ with the relevant ordinal notation, which in this
paper is the notation $a\in P$ such that $H_a = \emptyset_{(\alpha)}$.
(Thus infinite $\alpha$ are identified with the $a$ such that $|a|_\kO
= \alpha + 1$).  This choice greatly simplifies the presentation in
Section \ref{sec3} by removing the need to explicitly and constantly
deal with the non-uniformity between the finite and the infinite
discussed here.

As we are in the business of establishing the
$\Pi_\alpha$-completeness of various sets, we will construct
reductions to and from $\emptyset_{(\alpha)}$ for various values of $\alpha$.  All of our
reductions will be either to some $\emptyset_{(\alpha)}$ or to index sets.  Since all
sets of these kinds permit padding, it will suffice to find many-one
reductions, and this is what we do.  We use the technique of effective
transfinite recursion which is described in detail in \cite{sacks}.
For our purposes it can be stated as follows:

\begin{thm}
  Let $I:\omega\rightarrow\omega$ be a recursive function, and
  suppose for all $e\in\mathbb{N}$ and all $x\in \mathcal{P}$,
  if $\phi_e(y)$ is defined for all $y\in \mathcal{P}$ such that $y<_\kO x$, then $\phi_{I(e)}(x)$ is
  defined.  Then for some $c$, $\phi_c(x)$ is defined for all $x\in \mathcal{P}$, and $\phi_c (x) = \phi_{I(c)}(x)$ for any $x$ on
  which either converges.
\end{thm}

When we use this technique, the function $I$ will be defined only implicitly.

\subsection{Kechris and Woodin's differentiability rank}\label{sec2.3}
Kechris and Woodin \cite{kw} define a
rank on differentiable $C[0,1]$ functions as follows.  Let $\Delta_f(x,y)$ denote the
secant slope $$\Delta_f(x,y) = \frac{f(x)-f(y)}{x-y}.$$  They define a ``derivative'' operation, which is given below.  This operation starts with a closed set of points $P$ and removes from it some points at which $f$ seems to be differentiable.  A point $x$ is removed if the oscillation of $f'$ near $x$ is no more than the given $\varepsilon$.
\begin{definition}\label{kwdefinition} Given a closed set $P$, a function $f\in C[0,1]$ and $\varepsilon>0$,
  \begin{multline*} P'_{f,\varepsilon} = \{ x \in P : \forall \delta>0
    \exists p<q, r<s \in B(x, \delta)\cap[0,1] \\ \hbox{ with }
    [p,q]\cap[r,s]\cap P \neq \emptyset \hbox{ and } |\Delta_f(p,q) -
    \Delta_f(r,s)| \geq \varepsilon \}\end{multline*} where all the
  quantifiers range over rational numbers.  \end{definition} If $P$ is
closed, then $P'$ is closed as well, so for each $f\in C[0,1]$ and each $\varepsilon>0$ one defines the following
inductive hierarchy:
\begin{align*}
P^0_{f,\varepsilon} &= [0,1]\\ P_{f,\varepsilon}^{\alpha+1} &= (P_{f,\varepsilon}^\alpha)_{f,\varepsilon}' \\ P^\lambda_{f,\varepsilon} &= \cap_{\alpha<\lambda} P_{f,\varepsilon}^\alpha \hbox{ for a limit } \lambda
\end{align*}
Kechris and Woodin showed that for any $f\in C[0,1]$, $f$ is
differentiable if and only if \mbox{$\forall n \exists \alpha
  <\omega_1 (P_{f,1/n}^\alpha = \emptyset)$.}  Considering the supremum of
all such $\alpha$, they make the following definition:
\begin{definition}
  For each differentiable $f\in C[0,1]$, define
  $|f|_{KW}$ to be the least ordinal $\alpha$ such that $\forall
  \varepsilon P_{f,\varepsilon}^\alpha = \emptyset$.
\end{definition}

For example, if $f$ is any continuously differentiable function, then
$|f|_{KW}=1$, the least possible.  To see that $P^1_{f,\varepsilon} =
\emptyset$ for any such $f$ and any $\varepsilon$, let $\delta$ be
s.t. $|f'(z)-f'(y)|<\varepsilon$ whenever $|z-y|<\delta$.  Then for
any $x$ and any $p<q,r<s \in B(x,\delta/2)$, the Mean Value Theorem
provides $y\in[p,q]$ and $z\in[r,s]$ such that $f'(y)=\Delta_f(p,q)$
and $f'(z)=\Delta_f(r,s)$, so
$|\Delta_f(p,q)-\Delta_f(r,s)|<\varepsilon$ and $x\notin
P^1_{f,\varepsilon}$.  A common example of a differentiable function
whose derivative is not continuous is $x^2\sin(1/x)$, and this
function has differentiability rank 2.

\subsection{Basic building blocks}\label{buildingblocks}\label{sec2.4}

Kechris and Woodin show that for each ordinal $\alpha$, there is a
function with rank $\alpha$, and in order to show this they construct
an explicit $f$ with that rank.  This section gives a summary of the building blocks that they used to produce an example of a function living at each level of their hierarchy.  We will use the same building blocks in a more complicated construction in Section \ref{sec3}.

The most natural way of constructing a function while controlling its rank is to build it up recursively from smaller pieces.  Our basic building block is a simple continuously differentiable bump (Figure \ref{bumpfig}).

\begin{figure}
\begin{center}
\begin{tabular}{cc}
\includegraphics[width=\figwidth]{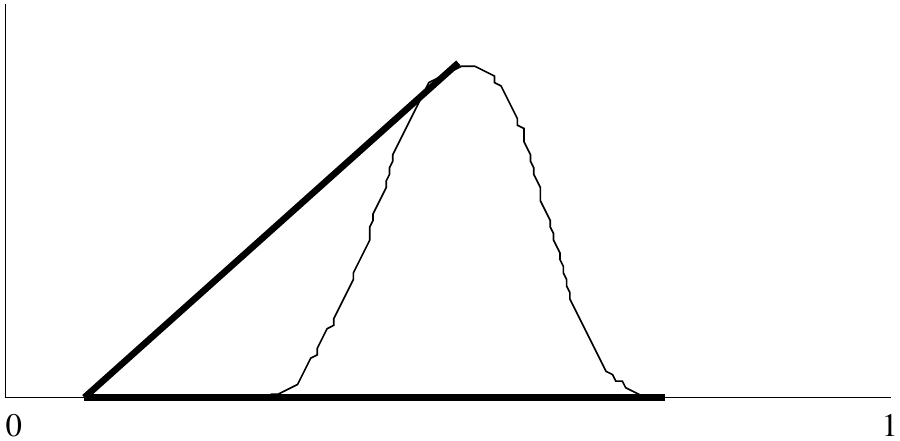} & 
\includegraphics[width=\figwidth]{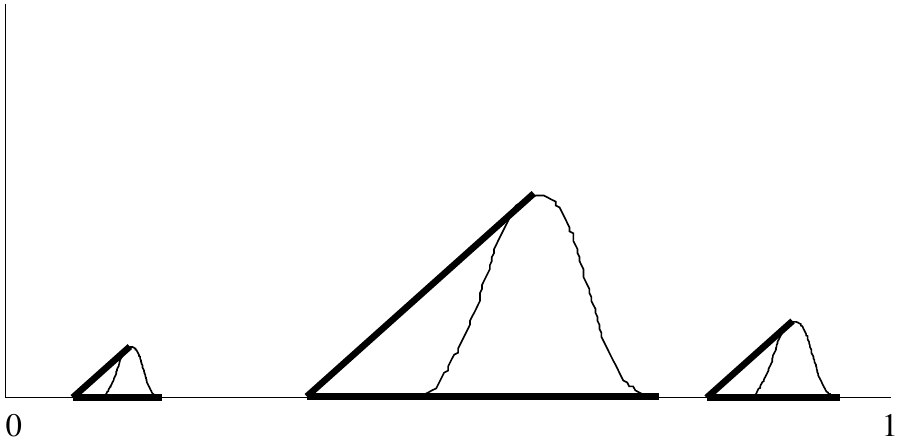} \\ (a) & (b)
\end{tabular}
\end{center}
\caption{(a) A continuously differentiable bump with one secant of slope zero and one secant of positive slope.  (b) Resized copies of this bump with proportions preserved.}
\label{bumpfig}
\end{figure}

Observe a certain pair of secants made by the existence of the bump,
one with slope zero and one with positive slope.  We build functions
out of resized copies of this same bump, always preserving the
proportions to keep the corresponding slopes uniform.  In Definition
\ref{kwdefinition} there is a free parameter $\varepsilon$, and one
compares various secants to see if their slope difference is at least
$\varepsilon$.  Therefore, by choosing a single sufficiently small
value for $\varepsilon$, all the secant pairs induced by the bumps
are made visible for the purposes of the rank-ascertaining process.  We
will sometimes refer to $\varepsilon$ as the \emph{oscillation
  sensitivity} because it sets the threshold above which oscillations
in the value of the derivative matter.

A simple rank 2 function is pictured in Figure \ref{rank2fig}.  To keep $0$ from being removed at the first iteration, we put a bump (and thus a disagreeing pair of secants) in every neighborhood of $0$.  To ensure the function remains differentiable at $0$ despite all the oscillation, we make the bumps small enough to fit inside an envelope of $x^2$.  The resulting rank 2 function can itself be proportionally shrunk and used as a building block in functions of larger rank.

\begin{figure}
\begin{center}
\begin{tabular}{cc}
\includegraphics[width=\figwidth]{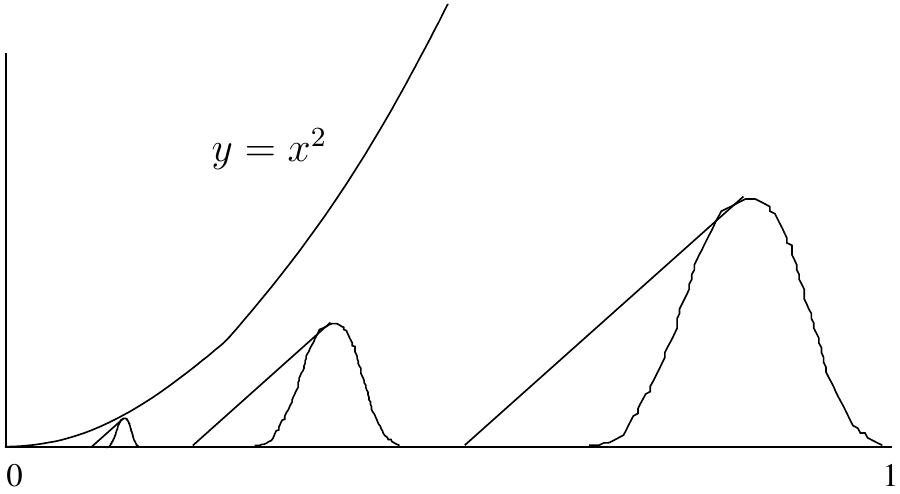} & 
\includegraphics[width=\figwidth]{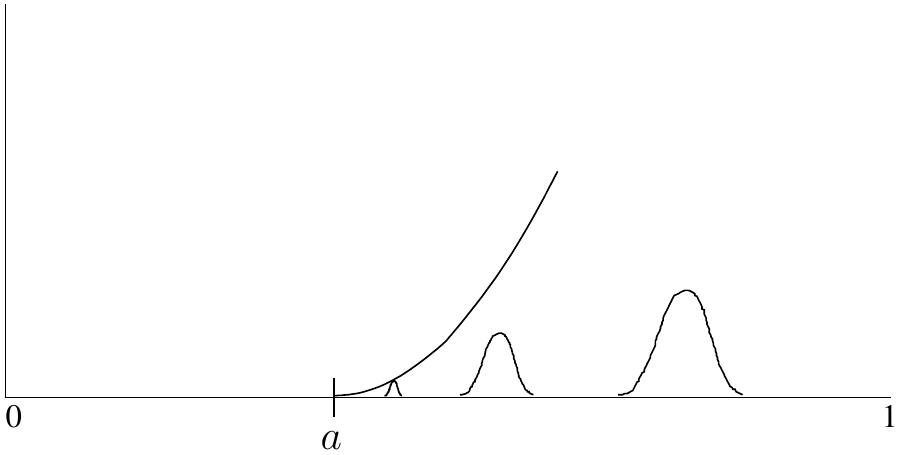} \\ (a) & (b)
\end{tabular}
\end{center}
\caption{(a) A simple differentiable function of rank 2.  (b) A shifted and resized copy of this function, which fits in a small neighborhood of the point $a$ and keeps $a$ alive through the first iteration.}
\label{rank2fig}
\end{figure}

The reason $0$ is removed at the second iteration, despite infinitely
many pairs of disagreeing secants, is that $P^1$ contains no points
which lie in the intersection $[p,q]\cap [r,s]$, where $p,q,r,s$ are
the endpoints of the intervals defining the disagreeing secant pair as
shown in Figure
\ref{shadows}.  
But if we have a rank $\alpha+1$ function to use as a building block
(the rank must be a successor for reasons discussed below), we can
make $0$ survive the $(\alpha+1)$st iteration.  By putting a shrunken
copy of our rank $\alpha+1$ function in $[p,q]\cap[r,s]$ as shown in
Figure \ref{shadows}, we construct a function of rank $\alpha+2$.  We
say that we have put the rank $\alpha+1$ function in the \emph{shadow} of
each bump.  In fact, it would suffice to put a rank $\alpha + 1$
function in the shadow of infinitely many of the bumps, and this is
done later in the paper.

\begin{figure}
\begin{center}
\begin{tabular}{cc}
\includegraphics[width=\figwidth]{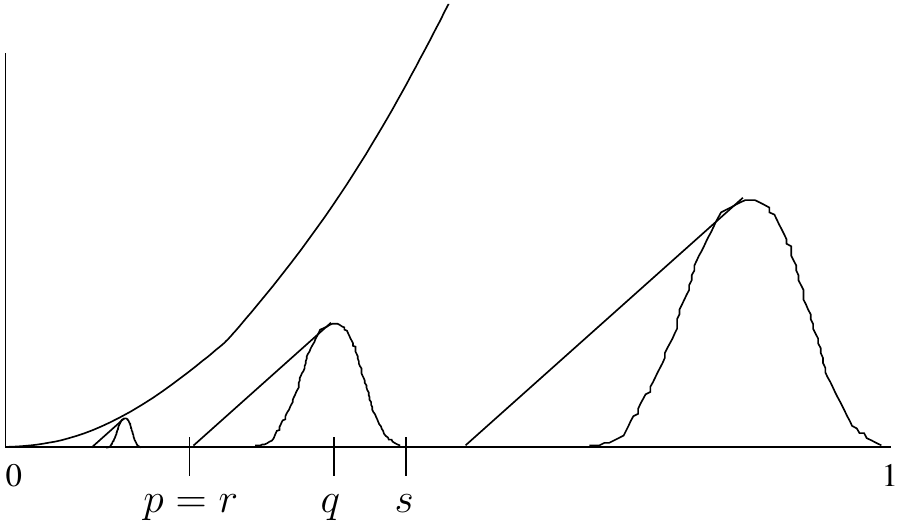} & 
\includegraphics[width=\figwidth]{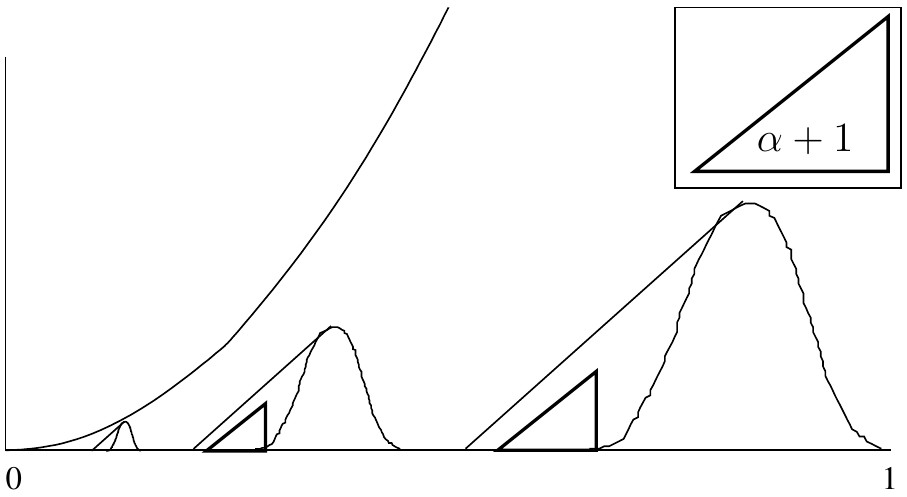} \\ (a) & (b)
\end{tabular}
\end{center}
\caption{(a) Points $p,q,r,s$ as used in Definition \ref{kwdefinition}. (b) A differentiable function of rank $\alpha+2$.  The triangle represents a function of rank $\alpha+1$.}
\label{shadows}
\end{figure}

Next we describe how to make functions of rank $\lambda+1$ and rank $\lambda$, where $\lambda$ is a limit ordinal.
We say that an oscillation sensitivity $\varepsilon$ \emph{witnesses} the rank of a function $f$ if $|f|_{KW}=\alpha$ and $P_\varepsilon^\beta \neq \emptyset$ for all $\beta<\alpha$.  Note that if a function has successor rank, there is always an $\varepsilon$ that witnesses this, but if the function has limit rank, there cannot be a witness.

Suppose we have a sequence of functions, with ranks cofinal in $\lambda$, whose ranks are all witnessed at a uniform sensitivity $\varepsilon$.  As shown in Figure \ref{limitfig}, a function of rank $\lambda + 1$ can be made by putting proportionally shrunken copies of functions of increasing rank in each neighborhood of $0$.  The rank of the resulting function is witnessed by the same $\varepsilon$.  

\begin{figure}
\begin{center}
\begin{tabular}{cc}
\includegraphics[width=\figwidth]{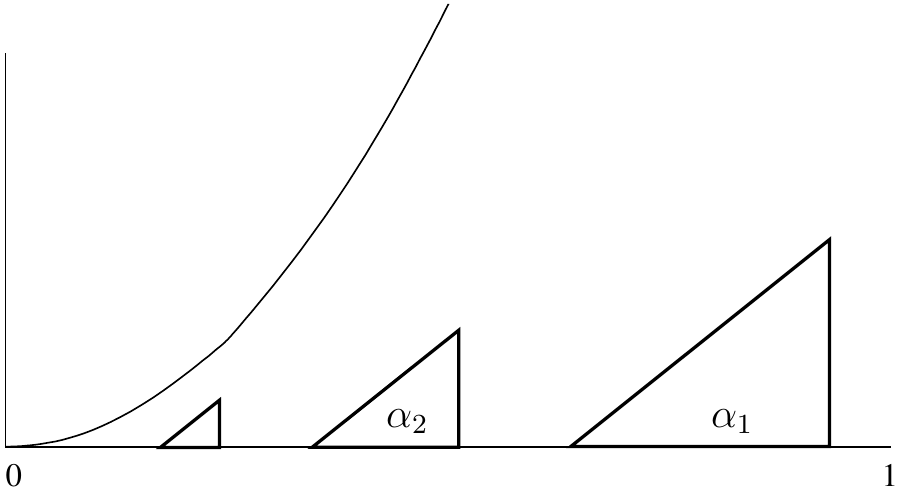} & 
\includegraphics[width=\figwidth]{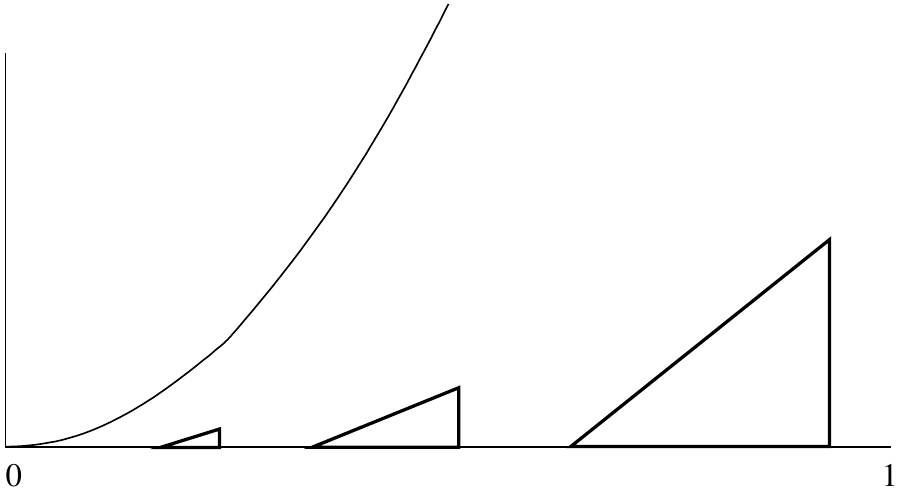} \\ (a) & (b)
\end{tabular}
\end{center}
\caption{(a) A function of rank $\lambda+1$ for $\lambda$ a limit ordinal. (b) A function of rank $\lambda$.}
\label{limitfig}
\end{figure}

By recursively applying the $\alpha+2$ step and the $\lambda+1$ step,
we can build functions of any successor rank.  To make a function of
rank $\lambda$, we must start with a sequence of functions with
uniformly bounded derivatives, whose ranks are cofinal in $\lambda$.
Because the derivatives are uniformly bounded, their possible secant
slope differences are also uniformly bounded by the Mean Value
Theorem.  Again we use shrunken copies of functions from the sequence,
but in addition to shrinking the $n$th function proportionally, we
also scale it vertically by a factor of $\frac{1}{n}$.  In the
resulting function, as $x$ approaches $0$ the nearby secant slope
differences approach zero, which has the effect of ensuring that $0$
is removed at the first iteration no matter what the oscillation
sensitivity.

Functions whose ranks are limit ordinals do not make good building blocks for more complicated functions because there is no $\varepsilon$ that witnesses their rank.  If we construct a rank $\lambda + 1$ function $f$, there needs to be a $\varepsilon$ such that $P^\lambda_{\varepsilon, f} \neq \emptyset$.  If we used a rank $\lambda$ function $g$ as a building block, then by compactness there would have to be some $\beta<\lambda$ such that $P^\beta_{\varepsilon,g}=\emptyset$.  So a function of rank $\beta$ would have been equally unhelpful.  That explains why, in our construction of the rank $\alpha+2$ function above, we needed to use a function with successor rank $\alpha+1$ as a building block.

\subsection{Notation}\label{sec2.5}
The following notations are used throughout.
\begin{definition}
  For each ordinal $\alpha$, let $D_\alpha$ denote the set of all
  indices $e$ such that $f_e\in C[0,1]$ is
  differentiable with $|f_e|_{KW}<\alpha$.  Define $D = \cup_{\alpha<\omega_1} D_\alpha$.
\end{definition}

  For any function $f\in C[0,1]$, we write $f[a,b]$ to denote the function
  which is identically 0 outside of $[a,b]$, and for $x\in [a,b]$,
  $f[a,b](x) = (b-a)f(\frac{x-a}{b-a})$.  Note that if $f$ is
  continuous and $f(0)=f(1)=0$, then $f[a,b]$ is continuous; it is
  computable when $f,a,$ and $b$ are and differentiable when $f$ is
  differentiable and $f'(0)=f'(1)=0$.

Similarly, for any real number $c\in[0,1]$ and any interval $[a,b]$,
let $c[a,b] = a + c(b-a)$.  This notation comes in handy when talking
about scaled down versions of functions, because $(b-a)f(c) =
f[a,b](c[a,b])$.  Also, this scaling preserves a function's
proportions ($f[a,b]'(c[a,b])= (b-a)f'(c)\frac{1}{b-a} = f'(c)$), so
$||f'|| = ||f[a,b]'||$ for any interval $[a,b]$.

\subsection{Facts about $D$}\label{sec2.6}

In section \ref{buildingblocks}, we described the major
components of Kechris and Woodin's construction of an explicit $f$
with $|f|_{KW}=\alpha$ for each $\alpha$.  When
$\alpha<\omega_1^{CK}$, their construction by transfinite recursion
easily effectivizes.  Therefore their argument also shows that for
each constructive $\alpha$, there is a computable differentiable $f$
with rank $\alpha$.

On the other hand, every computable differentiable function has
constructive rank.  This follows from work in the same paper by
Kechris and Woodin.

\begin{definition}
Let $\mathcal{D}$ denote the set of differentiable functions in $C[0,1]$.
\end{definition}

\begin{definition}
For each function $f\in C[0,1]$ and each $\varepsilon\in\mathbb{Q}^+$, define a tree $S_f^\varepsilon$ on $A = \{\langle p,q\rangle\ : 0\leq p < q\leq 1 \hband p,q\in \mathbb{Q}\}$ as follows: \begin{multline*}(\langle p_1,q_1 \rangle,\dots,\langle p_n,q_n\rangle) \in S_f^\varepsilon \iff \forall i \leq n (q_i-p_i \leq 1/i) \hband \cap_{i=1}^n [p_i,q_i]\neq \emptyset \\ \hband \forall i<n (|\Delta_f(p_{i+1},q_{i+1})-\Delta_f(p_i,q_i)|\geq\varepsilon).\end{multline*}
\end{definition}

Kechris and Woodin showed that for all $f\in C[0,1]$,
$f\in\mathcal{D}$ if and only if $\forall \varepsilon \in \mathbb{Q}^+
(S^\varepsilon_f \hbox{ is well-founded })$.  That makes possible the
following alternative rank definition:

\begin{definition}
  Let $f\in \mathcal{D}$.  Define $|f|^* = \sup
  \{|S^\varepsilon_f|+1: \varepsilon \in \mathbb{Q}^+\}$.
\end{definition}

\begin{lem} If $f\in \mathcal{D}$ is computable, then
  $|f|^*$ is constructive.
\end{lem}
\begin{proof}Note that the tree $S^\varepsilon_f$ would be computable if one did not have to verify that $|\Delta_f(p_{i+1},q_{i+1})-\Delta_f(p_i,q_i)|\geq\varepsilon$, a $\Pi_1$ statement.  In fact such a strong statement is not needed, and to get around it we use a computable approximation. For any
  computable $g\in C[0,1]$, rational $p\in[0,1]$, and rational $\delta
  >0$, the notation $[g(p)]_{\delta}$ refers to a standard
  $\delta$-approximation of $g(p)$, which is a rational number $z$
  such that $|g(p)-z|<\delta$.  (For specificity we could say
  $[g(p)]_\delta$ is the $b$ component of the smallest $\langle n, p,
  r, b, \delta/2\rangle$ in the computable real that encodes $g$.)
  Given a computable $f$, consider the following collection of trees
  $\tilde{S}^\varepsilon_f$, which are the same as the
  $S^\varepsilon_f$ defined above, except for the use of a computable
  approximation:
  \begin{multline*}(\langle p_1,q_1 \rangle,\dots,\langle
    p_n,q_n\rangle) \in \tilde S_f^\varepsilon \iff \forall i \leq n
    (q_i-p_i \leq 1/i) \hband \cap_{i=1}^n [p_i,q_i]\neq \emptyset \\
    \hband \forall i<n
    ([|\Delta_f(p_{i+1},q_{i+1})-\Delta_f(p_i,q_i)|]_{\varepsilon/4}\geq\varepsilon).\end{multline*}
  The $\tilde S_f^\varepsilon$ are computable trees, uniformly in $f$
  and $\varepsilon$.  Furthermore, for each $\varepsilon$,
  $S_f^{2\varepsilon} \subseteq \tilde S_f^\varepsilon \subseteq
  S_f^{\varepsilon/2}$, so $|S_f^{2\varepsilon}| \leq |\tilde
  S_f^\varepsilon| \leq |S_f^{\varepsilon/2}|$.  Therefore, although
  $|f|^*$ is defined in terms of $S_f^\varepsilon$, it is also true
  that $|f|^* = \sup\{|\tilde S_f^{\varepsilon}|+1 :\varepsilon
  \in\mathbb{Q}^+\}$.  Since $\tilde S_f^\varepsilon$ are defined
  uniformly in $\varepsilon$, the tree $$\tilde S =
  \{\varepsilon^\smallfrown \sigma : \varepsilon \in \mathbb{Q}^+,
  \sigma \in \tilde S_f^\varepsilon \}$$ is also computable, and
  $|f|^* = |\tilde S|$. Therefore $|f|^*$ is constructive.
\end{proof}

\begin{thm}[\cite{kw}]
  Let $f\in \mathcal{D}$.  Then if $f$ is linear,
  $|f|_{KW}=1$, and if $f$ is not linear, $|f|^* = \omega |f|_{KW}$.
\end{thm}

Therefore, for each computable $f$, $|f|_{KW} \in \kO$.  Thus
$$D=\bigcup_{\alpha<\omega_1^{CK}} D_\alpha.$$
By the standard definition of differentiability, $D$ is a $\Pi_1^1$
set.  Mazurkiewicz \nocite{maz} gave a reduction from well-founded
trees to differentiable functions.  This reduction, reproduced in
\cite{kw}, easily effectivizes, and therefore also serves as reduction
from $\kO$ to $D$.  Thus we know that $D$ is $\Pi_1^1$-complete.  We will
generate functions from well-founded trees using a method similar to
that of Mazurkiewicz.  By constructing the trees carefully we can
obtain finer grained results.

\section{Having rank at most $\alpha$ is $\Pi_{2\alpha+1}$}\label{sec2}

In this section, we show that ``$|f|_{KW}<\alpha+1$'' is a
$\Pi_{2\alpha+1}$ statement.  This
follows from a mostly straightforward translation of the definition of
differentiability rank into the formal language.  The only obstacle is
that the original definition needs to be slightly optimized.  In Section \ref{2.1} we give an equivalent definition of differentiability rank which uses fewer quantifiers.  In Section \ref{2.2} we formalize the sentence ``$|f|_{KW} \leq \alpha + 1$''.

\subsection{An equivalent rank function}\label{2.1}

In \cite{kw} the rank is defined using a ``derivative operation''
$P'_{f,\varepsilon}$ on sets $P$.  To prove our result we use an
almost identical operation $P^*_{f,\varepsilon}$ defined below.  The
only difference between this definition and the definition of
$P'_{f,\varepsilon}$ is that $\geq$ is replaced with $>$.  This is done in
order to make the statement $[0,1]^*_{f,\varepsilon} = \emptyset$ a
$\Sigma_2$ statement (instead of $\Sigma_3$), and this is necessary for
the base case of Proposition \ref{2n}.
\begin{definition} Given a closed set $P$, a function $f$ and $\varepsilon>0$,
\begin{multline*} P^*_{f,\varepsilon} = \{ x \in P : \forall \delta>0 \exists
 p<q, r<s \in B(x, \delta)\cap[0,1] \\ \hbox{ with }
 [p,q]\cap[r,s]\cap P \neq \emptyset \hbox{ and } |\Delta_f(p,q) -
 \Delta_f(r,s)| > \varepsilon \}\end{multline*} where all the
 quantifiers range over rational numbers.  \end{definition} 

It is easy to see that $P^*_{f,\varepsilon}$ is a closed subset of
$P$, so it makes sense to define a rank function using it.  We define
a hierarchy of closed sets analogously to \cite{kw}:

\begin{definition} \emph{($\tP^\alpha_{f,\varepsilon}(I)$ hierarchy)}
Fix a continuous function $f$, a rational $\varepsilon>0$, and a
closed set $I\subseteq [0,1]$.  
Define $\tP^0_{f,\varepsilon}(I)=I$.  Then for each ordinal $\alpha$,
define $\tP^{\alpha+1}_{f,\varepsilon}(I) = (\tP^\alpha_{f,\varepsilon}(I))^*_{f,\varepsilon}$.
If $\lambda$ is a limit ordinal, define 
$\tP^\lambda_{f,\varepsilon}(I) = \cap_{\alpha < \lambda}
\tP^\alpha_{f,\varepsilon}(I)$.
\end{definition}

In the special case $I=[0,1]$, we write $\tP^{\alpha}_{f,\varepsilon}$
instead of $\tP^{\alpha}_{f,\varepsilon}([0,1])$.  Sometimes the
function $f$ may also be omitted from the notation if it is clear from
context.

The rank of a differentiable function $f$ is defined in \cite{kw} to
be the smallest ordinal $\alpha$ such that for all $\varepsilon$,
$P^\alpha_{\varepsilon} = \emptyset$.  The next lemma shows our
$\tP^{\alpha}_{\varepsilon}$ hierarchy is similar enough to preserve
the notion.

\begin{lem}\label{interleaving} For any differentiable $f\in C[0,1]$, $\varepsilon>0$ and ordinal $\alpha$, $$\tP^\alpha_\varepsilon \subseteq P^\alpha_{\varepsilon/2}
  \subseteq \tP^\alpha_{\varepsilon/4}.$$
\end{lem}
\begin{proof}
The proof is by induction on $\alpha$.    When $\alpha = 0$ all these sets coincide.
Next we observe that both $'$ and $^*$ have the property that if
$P\subseteq Q$, then for any $\varepsilon$,
$P_\varepsilon'\subseteq Q_\varepsilon'\hbox{ and }P^*_\varepsilon \subseteq
Q^*_\varepsilon.$
Also it is
easy to observe that for all $\varepsilon$ and all $P$,
$P^*_\varepsilon \subseteq P'_{\varepsilon/2} \subseteq
P^*_{\varepsilon/4}.$
So  when $\alpha =
  \beta+1$, if we assume $\tP^\beta_\varepsilon \subseteq
  P^\beta_{\varepsilon/2} \subseteq \tP^\beta_{\varepsilon/4}$ we have
\begin{align*}
\tP^\alpha_\varepsilon &= (\tP^\beta_\varepsilon)^*_\varepsilon
\subseteq (P^\beta_{\varepsilon/2})^*_\varepsilon
\subseteq
(P^\beta_{\varepsilon/2})'_{\varepsilon/2} = P^\alpha_{\varepsilon/2}
\\ P^\alpha_{\varepsilon/2} &= (P^\beta_{\varepsilon/2})'_{\varepsilon/2} \subseteq (\tP^\beta_{\varepsilon/4})'_{\varepsilon/2}\subseteq (\tP^\beta_{\varepsilon/4})^*_{\varepsilon/4} =
\tP^\alpha_{\varepsilon/4}
\end{align*}
Finally, when $\lambda$ is a limit, $\cap_{\alpha < \lambda}
\tP^\alpha_\varepsilon \subseteq \cap_{\alpha<\lambda}
P^\alpha_{\varepsilon/2} \subseteq \cap_{\alpha<\lambda}
\tP^\alpha_{\varepsilon/4}$ follows because $\tP^\alpha_\varepsilon \subseteq P^\alpha_{\varepsilon/2}
  \subseteq \tP^\alpha_{\varepsilon/4}$ holds
for all $\alpha < \lambda$.  
\end{proof}

From Lemma \ref{interleaving} it is clear that for all $\alpha$, $$\forall \varepsilon
P^\alpha_\varepsilon = \emptyset \iff \forall \varepsilon
\tP^\alpha_\varepsilon = \emptyset,$$ and thus the notion of rank
defined using the $P^\alpha_\varepsilon$ hierarchy coincides with the
notion of rank defined using the $\tP^\alpha_\varepsilon$ hierarchy.

\subsection{The formal statements ``${\mathbf{|f|_{KW}\leq \alpha+1}}$''}\label{2.2}

Before we can use the previous section's definition to formalize
``$|f|_{KW}\leq\alpha + 1$'', we need the following lemma.  Briefly, the
lemma holds because membership in $\tP^\alpha_\varepsilon(I)$ is a
local property.

\begin{lem}\label{local1} Fix $f$ and $\varepsilon$.  For any closed $I\subseteq [0,1]$, any closed interval $[i,j]$, and any $\alpha$,
$$[i,j]\cap \tP^\alpha_{\varepsilon}(I) = \bigcap_{d>0}
  \tP^\alpha_{\varepsilon}([i-d,j+d]\cap I).$$
\end{lem}
\begin{proof}
On the one hand, suppose that $x\not\in
[i,j]\cap\tP^\alpha_{\varepsilon}(I)$.  If $x\not\in
[i,j]$ then eventually $x\not\in
[i-d,j+d]$.  So assume that $x\in[i,j]$.
  Then $x\not\in \tP^\alpha_{\varepsilon}(I)$, so $x$ could not be in
  $\tP^\alpha_{\varepsilon}([i-d,j+d]\cap I)$ for any $d$,
  since $\tP^\alpha_{\varepsilon}([i-d,j+d]\cap I) \subseteq
  \tP^\alpha_{\varepsilon}(I)$ for all $\alpha$.

For the other direction we proceed by induction on $\alpha$.  The
relationship certainly holds when $\alpha = 0$.  Suppose $\alpha =
\beta + 1$ and suppose that $x \in
      [i,j]\cap\tP^\alpha_{\varepsilon}(I)$.  We wish to show that $x
      \in \tP^\alpha_{\varepsilon}([i-d,j+d]\cap I)$, so fix
      $\delta$, and we will proceed to find our witnesses.  Since
      $x\in \tP^\alpha_{\varepsilon}(I)$, let $p<q,r<s \in B(x,
      \min(\delta,d/2))\cap I$ be such that $[p,q]\cap[r,s]\cap
      \tP^\beta_\varepsilon (I) \neq \emptyset$ and
      $|\Delta_f(p,q)-\Delta_f(r,s)|>
      \varepsilon$.  Then because $x\in[i,j]$, we have these same
      $p,q,r,s \in B(x,\delta)\cap[i-d,j+d]\cap I$, and in fact,
      because $p,q,r,s$ are within $d/2$ of $[i,j]$, we have
      $p,q,r,s \in [i-d/2,j+d/2]$.  If we can show that
      $[p,q]\cap[r,s]\cap \tP^\beta_{\varepsilon}([i-d,j+d]\cap I)
      \neq \emptyset$ then we are done.

Let $z\in [p,q]\cap[r,s]\cap\tP^\beta_\varepsilon(I)$.  By the
induction hypothesis,$$z\in \bigcap_{\zeta>0}
\tP^\beta_\varepsilon([\max(p,r)-\zeta,\min(q,s)+\zeta]\cap I).$$  So
in particular $$z \in \tP^\beta_\varepsilon([\max(p,r)-d/2,
  \min(q,s)+d/2]\cap I) \subseteq
\tP^\beta_\varepsilon([i-d,j+d]\cap I).$$  This completes the proof
for the successor case.

Finally, if $\alpha$ is a limit ordinal, we have
\begin{align*}[i,j]\cap \tP^\alpha_{\varepsilon}(I)  &= \bigcap_{\beta<\alpha}[i,j]\cap \tP^\beta_\varepsilon(I) \\ &= \bigcap_{\beta<\alpha} \bigcap_{d>0} \tP^\beta_\varepsilon([i-d,j+d]\cap I) \\ &= \bigcap_{d>0} \bigcap_{\beta<\alpha} \tP^\beta_\varepsilon([i-d,j+d]\cap I) \\
&= \bigcap_{d>0} \tP^\alpha_\varepsilon([i-d,j+d]\cap I).\end{align*}

\end{proof}

The definition of the rank of a function $f$ uses transfinite
recursion in order to calculate $P^\alpha_{f,\varepsilon}$ for each
$\alpha$ while holding $\varepsilon$ fixed.  Thus, knowing the
expressive complexity of ``$|f|_{KW} \leq 1$'' does not give us a
foothold into the expressive complexity of ``$|f|_{KW} \leq 2$'',
because ``$|f|_{KW} \leq \alpha$'' does not appear as a sub-expression
of ``$|f|_{KW} \leq \alpha+1$''.  The sub-expression which does
persist, and on which it is almost appropriate to transfinitely
recurse, is ``$[i,j]\cap \tP^\alpha=\emptyset$'', where $[i,j]$ is
some arbitrary interval.  Lemma \ref{local1} lets us express this
intersection in statements of the form
``$\tP^\alpha([i,j])=\emptyset$'', and so this last expression is a
useful core concept.  Its expressive complexity is
$\Sigma_{2\alpha}$, as seen in the next proposition.

\begin{prop}\label{2n}
Let $\alpha>0$ be a constructive ordinal, $\varepsilon, i, j \in \mathbb{Q}$ with $\varepsilon>0$ and $0\leq i < j \leq 1$.
The set of $e$ such that $\tP^\alpha_{f_e,\varepsilon}([i,j])=\emptyset$ is
$\Sigma_{2\alpha}$, uniformly in $\alpha, \varepsilon, i$ and $j$.
\end{prop}

\begin{proof}
  We carry along an arbitrary index $e$ and oscillation
  sensitivity $\varepsilon$, so to reduce clutter we write $f$ instead of $f_e$, and $\tP^\alpha$ instead of $\tP^\alpha_{f,\varepsilon}$.  

  In general, when $\alpha =
  \beta+1$, \begin{multline*}\tP^\alpha([i,j]) = [i,j] \setminus
    \bigcup\{ I: \forall p,q,r,s \in I \\ \left([p,q]\cap[r,s]\cap
      \tP^\beta([i,j]) = \emptyset \vee
      |\Delta_f(p,q)-\Delta_f(r,s)|\leq \varepsilon
    \right)\}\end{multline*} where $I$ ranges over intervals open in
  $[i,j]$.  Since the $I$ are closed under taking subsets, it suffices
  to let $I$ range over intervals open in $[i,j]$ with rational
  endpoints.  So $\tP^{\alpha}([i,j]) = \emptyset$
  if and only if these $I$ cover $[i,j]$.  If the $I$ do cover, then
  by compactness there is a rational $\delta$ such that for all
  $x\in[i,j]$, $B(x,\delta)\subseteq I$ for some $I$.  Thus there is
  a $\delta$ such that for any open interval $U$ with rational
  endpoints where diam$(U)<\delta$, $U\subseteq I$ for some
  $I$.  On the other hand, if the $I$ do not cover, then there cannot
  be any such $\delta$.  Thus if $\alpha = \beta + 1$,
  \begin{multline*} \tP^{\alpha}([i,j]) = \emptyset \iff \exists
    \delta>0 \forall c\in [i,j] \forall p,q,r,s\in
    B(c,\delta)\cap[i,j] \\ \left([p,q]\cap[r,s]\cap
      \tP^{\beta}([i,j]) = \emptyset \vee
      |\Delta_f(p,q)-\Delta_f(r,s)|\leq \varepsilon
    \right)\end{multline*} where all quantifiers range over the
  rationals.

When $\beta=0$, $[p,q]\cap[r,s]\cap \tP^{\beta}([i,j]) = \emptyset
\Leftrightarrow [p,q]\cap[r,s]\cap[i,j]=\emptyset$, so the above statement is
$\Sigma_2$ uniformly in $e,\varepsilon,i,$ and $j$.

When $\beta>0$, we have
\begin{multline*}[p,q]\cap[r,s]\cap \tP^{\beta}([i,j]) =
\emptyset \\ \iff \exists \zeta \tP^{\beta}([\max(p,r)-\zeta, \min(q,s)+\zeta]\cap [i,j]) = \emptyset.
\end{multline*}
which follows from Lemma \ref{local1} and compactness.  Thus with the assumption that
$\tP^{\beta}([c,d])=\emptyset$ is $\Sigma_{2\beta}$ uniformly in all variables, 
then
$\tP^{\beta+1}([i,j])=\emptyset$ is $\Sigma_{2\beta + 2}$, uniformly in all variables.

Finally, suppose that $\alpha$ is a limit, given as a uniform supremum $\alpha = \sup_n \beta_n$.  Then by compactness and the
definition of $\tP^\alpha$ for $\alpha$ a limit, 
$$P^\alpha([i,j]) = \emptyset \iff \exists n
  \tP^{\beta_n}([i,j]) = \emptyset.$$ So assuming
that $\tP^{\beta_n}([i,j]) = \emptyset$ is uniformly $\Sigma_{2\beta_n}$ in all variables including $n$, we see that
$\tP^{\alpha}([i,j])=\emptyset$ is uniformly $\Sigma_\alpha$, which is the same as $\Sigma_{2\alpha}$ since $\alpha$ is a limit.\end{proof}

\begin{prop}\label{prop2}
For any constructive $\alpha>0$, $D_{\alpha+1}$ is $\Pi_{2\alpha+1}$, uniformly in $\alpha$.
\end{prop}

\begin{proof} We have
$$e\in D_{\alpha+1} \iff f_e\in C[0,1] \wedge \forall
    \varepsilon [\tP^{\alpha}_{f_e,\varepsilon} = \emptyset]$$ where
  $\varepsilon$ ranges over positive rationals.  Recall that ``$f_e\in C[0,1]$'' is $\Pi_2$, and by Proposition \ref{2n},  $\tP^{\alpha}_{f_e,\varepsilon} = \emptyset$ is $\Sigma_{2\alpha}$.
Thus the right hand side is a
  $\Pi_{2\alpha+1}$ statement, uniformly in $\alpha$ and $e$.
\end{proof}

\section{Having rank at most $\alpha$ is $\Pi_{2\alpha+1}$-complete}\label{sec3}
In this section, we provide a many-one reduction in the other
direction, from $\overline{\emptyset_{(2\alpha + 1)}}$ to
$D_{\alpha+1}$.  The most significant step of the reduction is found in Section \ref{sec3.3}.  To set up this step, we
happen to need only a certain class of $C[0,1]$ functions which can be
structurally represented by well-founded trees according to a
recipe reminiscent of Mazurkiewicz's original reduction.  This allows
us to construct a function of the right rank through an intermediate
step of constructing a tree with the right structure.

In Section \ref{sec3.1} we construct special $C[0,1]$ functions which
reflect the structure of well-founded trees on
$\mathbb{N}^{<\mathbb{N}}$.  In Section \ref{sec3.2}, we define a rank
on well-founded trees which agrees with the differentiability rank of
the functions that the trees generate, when a fixed arbitrary value for $\varepsilon$ is used.  In Section \ref{sec3.3} we
give a reduction from canonical complete sets to trees of
an appropriate rank, obtaining a result one jump short of the final result.  Section \ref{sec3.4} combines the results of
the previous sections with the additional ingredient
of varying $\varepsilon$ to obtain the final result.

\subsection{Making differentiable functions out of well-founded trees}\label{sec3.1}

The idea of this section is to set up countably many closed disjoint
intervals in $[0,1]$, put the intervals in bijective correspondence
with $\mathbb{N}^{<\mathbb{N}}$, and then given a tree $T\subseteq
  \mathbb{N}^{<\mathbb{N}}$, define $f_T$ as a sum of continuously differentiable bumps supported on
  each of the intervals which correspond to $\sigma\in T$.  These functions are structurally similar to the ones described in Section \ref{buildingblocks}.  If $S = \{\rho : \sigma^\smallfrown \rho \in T\}$ then a shrunken version of $f_S$ can be found in $f_T$.  Furthermore, if $\tau \supset \sigma$, then the bump corresponding to $\tau$ is in the shadow of the bump corresponding to $\sigma$.  The
  intervals are arranged so that the resulting $f_T$ has a
  differentiability rank which can be computed from $T$ in a way that is
  described in the next section.

In the
following definition, the choices of the constants $\frac{1}{2}$ and
$\frac{1}{4}$ and the bounds on $p$ and $p'$ are arbitrary, but
consistent with each other.  The requirement $b_n-a_n <
(a_n-\frac{1}{4})^2$ is what keeps $f_T$ everywhere differentiable.

\begin{definition}\label{tree-translation}
Let $p:[0,1]\rightarrow
\mathbb{R}$ be a computable function satisfying
\begin{enumerate}
\item $p$ is continuously differentiable
\item $p(\frac{1}{2})=\frac{1}{2}$
\item $p(0)=p(1)=p'(0)=p'(1)=0$
\item $||p|| < 1$ and $||p'||<2$
\end{enumerate}
  Let $\{[a_n,b_n]\}_{n\in\mathbb{N}}$ be any computable
  sequence of intervals with rational endpoints satisfying
\begin{enumerate}
\item Each interval is contained in $(\frac{1}{4},\frac{1}{2})$
\item $b_{n+1}<a_n<b_n$ for each n.
\item $\lim_{n\rightarrow\infty}a_n = \frac{1}{4}$
\item $b_n-a_n < (a_n-\frac{1}{4})^2$ for each $n$
\end{enumerate}
Then for any well-founded tree $T\in\Nat^{<\Nat}$, define $f_T$ as
  follows.
\begin{enumerate}
\item If $T$ is empty, $f_T\equiv 0$.
\item Otherwise, $f_T = p[\frac{1}{2},1] + \sum_{n=0}^\infty f_{T_n}[a_n,b_n]$
\end{enumerate}
\end{definition}

Recall that $f[a,b]$ denotes a copy of $f$ proportionally resized to
have domain $[a,b]$, and that $T_n$ denotes $\{\sigma :
n^\smallfrown \sigma \in T\}$, the $n$th subtree of $T$.  Now we
verify that the above definition produces well-defined computable
differentiable functions.

\begin{prop}\label{goodfunction}
  For any well-founded computable tree $T\in\Nat^{<\Nat}$:
\begin{enumerate}
\item $f_T$ is uniformly computable in $T$
\item $f_T$ is differentiable
\item $f_T(0)=f_T(1)=f_T'(0)=f_T'(1)=0$
\item $||f_T||<1$ and $||f_T'||<2$
\end{enumerate}
\end{prop}
\begin{proof}
  Proceeding by induction on the rank of the tree, in the base case
  all four properties are satisfied.  Assume they hold for all trees
  of rank less than $|T|$.  Then the sequence $f_{T_n}$ is uniformly
  computable with each $||f_{T_n}||<1$.  Then on any interval whose
  closure does not contain $\frac{1}{4}$, $f_T$ is equal to a
  uniformly determined finite sum of computable functions, and is thus
  computable.  And for $\varepsilon$ sufficiently small, we claim that
  $|f_T\left((\frac{1}{4}-\varepsilon,\frac{1}{4}+\varepsilon)\right)|<
  \varepsilon^2$.  This follows because $||f_{T_n}||<1$ by induction and because the
  intervals $[a_n,b_n]$ are disjoint, so for any $x$ in such an
  interval we have the bound
  $$|f_T(x)| = f_{T_n}[a_n,b_n](x) \leq ||f_{T_n}||(b_n-a_n) \leq 1\cdot
  (a_n-\frac{1}{4})^2 \leq (x-\frac{1}{4})^2<\varepsilon^2.$$
  Therefore $f_T$ is uniformly computable in $T$.  Similarly, assuming
  $f_{T_n}$ are each differentiable with
  $f_{T_n}(0)=f_{T_n}(1)=f_{T_n}'(0)=f_{T_n}'(1)=0$, then each
  $f_{T_n}[a_n,b_n]$ is differentiable.  Then $f_T$ is certainly
  differentiable at any point $x\neq\frac{1}{4}$, since on some
  neighborhood of that point $f_T$ is equal to a finite sum of
  differentiable functions.  On the other hand, in the vicinity of
  $\frac{1}{4}$, $f_T$ satisfies $|f_T(x)|\leq (x-\frac{1}{4})^2$, so
  $f_T$ is differentiable at $\frac{1}{4}$ as well.  Because
  \mbox{$f_T\uhr [0,\frac{1}{4}] \equiv 0$} and $p(1)=p'(1)=0$, we
  have $f_T(0)=f_T(1)=f'_T(0)=f_T'(1)=0$.  Finally, $||f_T||<1$ and
  $||f_T'||<2$ by induction, because $||p||<1$,$||p'||<2$, and
  $||f_{T_n}||<1$,$||f_{T_n}'||<2$, for each $n$, and the shrunken
  copies $p[\frac{1}{2},1]$ and $f_{T_n}[a_n,b_n]$ have disjoint support.
\end{proof}

We close this section with some comments about why this $f_T$ is
defined as it is, using the concepts from Section
\ref{buildingblocks}.  Note that for every nonempty $S$, $$
\Delta_{f_{S}}(0,\frac{3}{4}) = \frac{1}{3} \hband
\Delta_{f_{S}}(0,\frac{1}{2}) = 0.$$ Now for each $n$,
$f_{T_n}[a_n,b_n]$ is a proportionally shrunken copy of $f_{T_n}$, so
unless $T_n$ is empty, $f_{T_n}[a_n,b_n]$ contributes a bump and its
pair of secants with slopes 0 and $\frac{1}{3}$.  Thus a tree with
infinitely many children of the root has infinitely many pairs of
these disagreeing secants.  If we construct $T_n$ so that $f_{T_n}$ has
a large rank, the $n$th disagreeing pair of secants will be visible
for many iterations of the rank-ascertaining process for $f_T$,
because $P_{f_T}^\alpha \cap [a_n,\frac{a_n+b_n}{2}]$ will be
nonempty for many iterations.  If we construct $T$ so that $f_{T_n}$
has large rank for infinitely many $n$, these disagreeing pairs of
secants can make a contribution to raising the Kechris-Woodin rank of
$f_T$.  This can happen in two ways: if $|f_{T_n}|_{KW} = \alpha + 1$
for infinitely many $n$, then $|f_T|\geq\alpha+2$.  And if the ranks
of the $f_{T_n}$ are unbounded below a limit ordinal $\lambda$, then
$|f_T|_{KW}=\lambda + 1$.

\subsection{A rank on well-founded trees which agrees with the differentiability rank of the corresponding functions}\label{sec3.2}

We will now show that when a function is generated from a tree in the way described above,
its Kechris-Woodin rank can be read right off the tree.  Furthermore,
we will see that this function's rank can already be witnessed at a
fixed oscillation sensitivity $\varepsilon = \frac{1}{4}$.  That is,
the rank of $f_T$ is always a successor, and when
$|f_T|_{KW}=\alpha+1$, then $\tP_{f,\frac{1}{4}}^\alpha \neq
\emptyset$.  Here is a rank on trees which corresponds to the
differentiability rank of the functions they generate.
\begin{definition}\label{tree-rank}
  For a well-founded tree $T\in\Nat^{<\Nat}$, define the
  limsup rank of the tree by $$|T|_{ls} = \max\left( \sup_n
    |T_n|_{ls}, (\limsup_n |T_n|_{ls})+1 \right),$$ if $T$ is
  nonempty, and $|T|_{ls}=0$ if $T$ is empty.
\end{definition}
Note that reordering the subtrees does not change the limsup rank of the tree.
A node can have a rank higher than all its children in
one of two situations: either there is no child of maximal rank, or
there are infinitely many maximal rank children.  In the next
proposition, we will see that this mechanism corresponds exactly to
the mechanism for constructing functions of increasing
differentiability rank.

The following two straightforward lemmas which we will use later are
woven into the proof of Fact 3.5 in \cite{kw}.  For the purposes of
exposition, we state and prove them here.
\begin{lem}\label{A}
If $U\subseteq[0,1]$ is open and $f\uhr U = g\uhr U$, then for all $\alpha$ and $\varepsilon$, $P_{f,\varepsilon}^\alpha \cap U = P_{g,\varepsilon}^\alpha \cap U$.
\end{lem}
\begin{proof}
  By induction on $\alpha$.  The base and limit cases are trivial.
  Suppose that $P^\alpha_{f,\varepsilon}\cap U = P^\alpha_{g,\varepsilon}\cap U$.
  Fix $x\in U$ and let $\lambda$ be small enough that
  $B(X,\lambda)\subseteq U$.  Then $x\in P^{\alpha+1}_{f,\varepsilon}$
  if and only if for all $\delta <\lambda$ there are $p,q,r,s\in
  B(x,\delta)$ such that
  $|\Delta_f(p,q)-\Delta_f(r,s)|\geq\varepsilon$ and
  $[p,q]\cap[r,s]\cap P^\alpha_{f,\varepsilon} \neq \emptyset$.  Since
  $p,q,r,s\in B(x,\delta)\subseteq B(x,\lambda) \subseteq U$, we have
  $[p,q]\cap[r,s]\cap P^{\alpha}_{f,\varepsilon} \neq \emptyset$ if
  and only if $[p,q]\cap[r,s]\cap P^{\alpha}_{g,\varepsilon} \neq
  \emptyset$.  Thus $x\in P^{\alpha+1}_{f,\varepsilon}$ if and only if
  $x\in P^{\alpha+1}_{g,\varepsilon}$.
\end{proof}

Recall that for any function $f\in C[0,1]$, we write $f[a,b]$ to
  denote a proportionally shrunken verson of $f$.  By definition,
  $f[a,b]$ is the function which is identically 0 outside of $[a,b]$,
  and for $x\in [a,b]$, $f[a,b](x) = (b-a)f(\frac{x-a}{b-a})$.
  Similarly, for any real number $c\in[0,1]$ and any interval $[a,b]$,
  let $c[a,b] = a + c(b-a)$.  The point is that $c$ is to $f$ as
  $c[a,b]$ is to $f[a,b]$.

\begin{lem}\label{B}
  Let $f\in C[0,1]$ be a differentiable function satisfying
  $f(0)=f(1)=f'(0)=f'(1)=0$.  Let $[a,b]\subseteq[0,1]$ be an interval
  with rational endpoints.  Then $|f|_{KW} = |f[a,b]|_{KW}$.
  Furthermore, for any ordinal $\alpha$ and for all $x\in[0,1]$,
\begin{enumerate}
\item $x\in P^\alpha_{f,\varepsilon} \implies x[a,b]\in P^{\alpha}_{f[a,b],\varepsilon}$
\item $x[a,b]\in P^\alpha_{f[a,b],\varepsilon} \implies x\in P^\alpha_{f,\varepsilon/2}$
\end{enumerate}
\end{lem}
\begin{proof}

  Proceeding by induction, it is clear that the both items holds when
  $\alpha=0$.  The limit case is also trivial.  

  Assume the first item holds for some $\alpha$.  If $x\in
  P^{\alpha+1}_{f,\varepsilon}$, then the collection of all the tuples
  $p,q,r,s$ which witness this can be mapped to a collection of tuples
  $p[a,b],q[a,b],r[a,b],s[a,b]$ which witness $x[a,b] \in
  P^{\alpha+1}_{f[a,b],\varepsilon}$.  
  That proves the
  first item.

  On the other hand, suppose the second item holds for some $\alpha$.
  If $x[a,b]\in P^{\alpha + 1}_{f[a,b],\varepsilon}$ and $x \in (0,1)$
  (i.e. $x$ is not an endpoint), then as above corresponding witnesses
  can always be chosen for sufficiently small neighborhoods of $x$, so
  $x\in P^{\alpha+1}_{f,\varepsilon} \subseteq
  P^{\alpha+1}_{f,\varepsilon/2}$.  Last we consider the endpoint case: suppose $a \in
  P^{\alpha+1}_{f[a,b],\varepsilon}$ (and the case $b\in
  P^{\alpha+1}_{f[a,b],\varepsilon}$ is of course just the same).
  Because $a = 0[a,b]$ and $f'(0)=0$, let $\lambda$ be small enough
  that for all distinct $p,q\in B(a,\lambda)$ with $p\leq a\leq q$,
  $|\Delta_{f[a,b]}(p,q)|<\varepsilon/4$.  Then for each $\delta>0$,
  there are $p,q,r,s\in B(a,\min(\lambda, (b-a)\delta))$ such that
  $|\Delta_{f[a,b]}(p,q)-\Delta_{f[a,b]}(r,s)|\geq \varepsilon$ and
  $[p,q]\cap[r,s]\cap P^\alpha_{f[a,b],\varepsilon} \neq \emptyset$.
  Then without loss of generality, $|\Delta_{f[a,b]}(p,q)|\geq
  \varepsilon/2$, so $a<p<q$.  If also $a<r<s$, then we are done since
  the corresponding $\frac{p-a}{b-a},$ etc. can be used as the witness
  for $\delta$.  It is impossible that $r<s<a<p<q$ because
  $[p,q]\cap[r,s]\neq\emptyset$.  In the last case, if $r<a<s$, this
  implies that $|\Delta_{f[a,b]}(r,s)|<\varepsilon/4$, so
  $|\Delta_{f[a,b]}(p,q)|\geq 3\varepsilon/4$.  But then also
  $|\Delta_{f[a,b]}(a,s)|<\varepsilon/4$, and thus
  $|\Delta_{f[a,b]}(p,q)-\Delta_{f[a,b]}(a,s)|\geq \varepsilon/2$.
  Also $[p,q]\cap[a,s] = [p,q]\cap[r,s]$, and there is some $y\in
  [p,q]\cap[a,s]\cap P^\alpha_{f[a,b],\varepsilon}$, and by induction
  $\frac{y-a}{b-a}\in P^\alpha_{f,\varepsilon/2}$.  Therefore
  $\frac{p-a}{b-a},\frac{q-a}{b-a},0,\frac{s-a}{b-a}$ will do, and
  thus $x\in P^{\alpha+1}_{f,\varepsilon/2}$.

  Finally, note that by the previous lemma, $P^\alpha_{f[a,b],\varepsilon}\cap
  ([0,1]\setminus[a,b]) = \emptyset$ for any $\alpha>0$.  Therefore, $|f|_{KW} = |f[a,b]|_{KW}$.
\end{proof}

The next proposition shows that for any well-founded tree $T$, the
differentiable function $f_T$ defined in the previous section has rank
$|f_T|_{KW}=|T|_{ls}$, and that the rank of $f_T$ is witnessed at
oscillation sensitivity $\varepsilon = \frac{1}{4}$.

\begin{prop}\label{tree-language}
  For any nonempty well-founded tree $T\in\Nat^{<\Nat}$,
\begin{enumerate}
\item  $|T|_{ls}$ is
  a successor,
\item The function $f_T$ is differentiable with
  $|f_T|_{KW} = |T|_{ls}$, and
\item Letting $|T|_{ls}=\alpha+1$, we have
  $P_{f,\frac{1}{4}}^\alpha\neq \emptyset$.
\end{enumerate}
\end{prop}
\begin{proof}
  The proof is by induction on the usual rank of the tree.  If $T$ is
  just a root (smallest option for the rank of the the tree since the
  statement is for nonempty trees only) then $f_T$ is just
  $p[\frac{1}{2},1]$, so it is continuously differentiable with
  $|f_T|_{KW} = 1$.  For each $n$, $T_n = \emptyset$ so $|T_n|_{ls} =
  0$ so $\sup_n|T_n|_{ls} = \limsup_n|T_n|_{ls} = 0$, so $|T|_{ls}=1$.
  
If $T$ is more than a root, assume the lemma holds for each of the
  subtrees $T_n$.  First we show that $|f_T|_{KW} \geq
  |T|_{ls}$.  Fix $n$ and let $|T_n|_{ls} = \alpha + 1$.  Then
  by the inductive hypothesis $|f_{T_n}|_{KW} = \alpha + 1$ and
  $P_{f_{T_n},\frac{1}{4}}^\alpha \neq \emptyset$.  By Lemma
  \ref{B}, $x\in P^\alpha_{f_{T_n},\frac{1}{4}}\implies x[a,b] \in
  P^\alpha_{f_{T_n}[a_n,b_n],\frac{1}{4}}$, so
  $P^\alpha_{f_{T_n}[a_n,b_n],\frac{1}{4}}\neq \emptyset$.  Because
  the $[a_n,b_n]$ are closed and disjoint from each other and from
  $[\frac{1}{2},1]$, there is an $\varepsilon>0$
  such that $f_T\uhr(a_n-\varepsilon,b_n+\varepsilon) = f_{T_n}[a_n,b_n]\uhr(a_n-\varepsilon,b_n+\varepsilon)$, and therefore using Lemma \ref{A},
  $P_{f_T,\frac{1}{4}}^\alpha \cap (a_n-\varepsilon,b_n+\varepsilon) =
  P_{f_{T_n}[a_n,b_n],\frac{1}{4}}^\alpha \cap (a_n-\varepsilon, b_n+\varepsilon)\neq \emptyset$.  Therefore $P_{f_T,\frac{1}{4}}^\alpha \neq
  \emptyset$ and thus $|f_T|_{KW}\geq\alpha+1$. So $|f_T|_{KW} \geq
  \sup_n|T_n|_{ls}$.

  Now let us show that $|f_T|_{KW} \geq (\limsup_n|T_n|_{ls}) + 1$.
  Let $\alpha = \limsup_n|T_n|_{ls}$.  We will show that
  $\frac{1}{4}\in P_{f_T,\frac{1}{4}}^\alpha$.  First we show that for any
  $\beta<\alpha$, $\frac{1}{4}\in P_{f_T,\frac{1}{4}}^\beta$.  There
  are infinitely many $n$ such that $|T_n|_{ls}>\beta$, so $
  P_{f_{T_n},\frac{1}{4}}^\beta\neq\emptyset$ for infinitely many $n$
  by the inductive hypothesis, so
  $P^\beta_{f_{T_n}[a_n,b_n],\frac{1}{4}}\neq \emptyset$ for
  infinitely many $n$ by Lemma \ref{B}.  By Lemma \ref{A}, $
  P_{f_{T_n}[a_n,b_n],\frac{1}{4}}^\beta \subseteq
  P_{f_T,\frac{1}{4}}^\beta$.  Because $\lim_{n\rightarrow\infty} a_n
  = \frac{1}{4}$, and infinitely many $[a_n,b_n]$ contain an element
  of $P^\beta_{f_T,\frac{1}{4}}$, $\frac{1}{4}$ is a limit point of $
  P_{f_T,\frac{1}{4}}^\beta$.  Because this set is closed,
  $\frac{1}{4}$ must be in it as well.  Thus $\frac{1}{4} \in
  P_{f_T,\frac{1}{4}}^\beta$ for all $\beta<\alpha$.  If $\alpha$ is
  as limit, this implies $\frac{1}{4} \in P_{f_T,\frac{1}{4}}^\alpha$,
  so $|f_T|_{KW}>\alpha$ if $\alpha$ is a limit.  Now suppose $\alpha$
  is a successor.  Let $\alpha = \beta + 1$.  Let $U$ be a
  neighborhood of $\frac{1}{4}$, and let $n$ be chosen such that
  $[a_n,b_n]\subseteq U$ and $
  P_{f_{T_n}[a_n,b_n],\frac{1}{4}}^\beta\neq\emptyset$.  Then
  $\Delta_{f_T}(a_n,\frac{3}{4}[a_n,b_n]) = \frac{1}{3}$ and
  $\Delta_{f_T}(a_n,\frac{1}{2}[a_n,b_n]) = 0$, and $
  [a_n,\frac{3}{4}[a_n,b_n]]\cap[a_n,\frac{1}{2}[a_n,b_n]]\cap
  P_{f_T,\frac{1}{4}}^\beta\neq\emptyset$.  Therefore $\frac{1}{4}\in
  P_{f_T,\frac{1}{4}}^{\beta+1}$, so again $|f_T|_{KW}>\alpha$.  This
  completes the claim that $|f_T|_{KW}\geq|T|_{ls}$.

  Now let us show that $|f_T|_{KW} \leq |T|_{ls}$.  First, let $\alpha
  = \sup_n |T_n|_{ls}$.  Note that $\alpha>0$ because the case of $T$
  being only a root was already considered separately.  
  For each $n$, $|T_n|_{ls}\leq \alpha$, so by induction $|f_{T_n}|_{KW}
  =|f_{T_n}[a_n,b_n]|_{KW}\leq \alpha$.  So for each $n$ and
  $\varepsilon$ we have $P^\alpha_{f_{T_n}[a_n,b_n],\varepsilon} =
  \emptyset$ and also $P^\alpha_{p[\frac{1}{2},1],\varepsilon} =
  \emptyset$.  Cover $[0,1]\setminus\{\frac{1}{4}\}$ with open
  intervals $U$ such that each $U$ intersects at most one of the
  $[a_n,b_n]$ or $[\frac{1}{2},1]$.  Then for each such interval and
  each $\varepsilon$, $P^\alpha_{f_T,\varepsilon}\cap
  U=P^\alpha_{f_{T_n}[a_n,b_n],\varepsilon}\cap U = \emptyset$, or
  $P^\alpha_{f_T,\varepsilon}\cap U =
  P^\alpha_{p[\frac{1}{2},1],\varepsilon}\cap U=\emptyset$,
  respectively.  Therefore, for all $\varepsilon$, $
  P_{f_T,\varepsilon}^{\alpha}\subseteq \{\frac{1}{4}\}$.  If
  $\limsup_n|T_n|_{ls} = \sup_n|T_n|_{ls}$, then $|T|_{ls}=\alpha+1$,
  so this is enough: $ P_{f_T,\varepsilon}^{\alpha+1}=\emptyset$ for
  all $\varepsilon$.

  On the other hand, suppose $\limsup_n |T_n|_{ls} < \sup_n
  |T_n|_{ls}$.  Then $\alpha=|T|_{ls}=\sup_n |T_n|_{ls}$ is a
  successor, because the induction hypothesis guarantees $|T_n|_{ls}$
  is always a successor, and therefore if the sup were a limit, it
  would be equal to the limsup.  Let $\alpha=\beta+1$.  Then
  eventually $|T_n|_{ls}\leq \beta$.  Let $V$ be an open neighborhood
  of $\frac{1}{4}$ such that $[a_n,b_n]\cap V\neq \emptyset$ implies
  $|T_n|_{ls}\leq \beta$.  Covering $V\setminus\{\frac{1}{4}\}$ with
  open intervals $U$ as before, we find
  $P^\beta_{f_{T_n}[a_n,b_n],\varepsilon} \cap U = \emptyset$ for each
  such $U\subseteq V$ and each $\varepsilon$, so
  $P_{f_T,\varepsilon}^\beta\cap V \subseteq \{\frac{1}{4}\}$, so
  $P_{f_T,\varepsilon}^{\beta+1}\cap V = \emptyset$. Therefore $
  P_{f_T,\varepsilon}^{\beta+1}=\emptyset$.  Thus $|f_T|_{KW} \leq
  \sup_n |T_n|_{ls}$.
\end{proof}

\subsection{Recognizing trees of limsup rank $\alpha$ is $\Sigma_{2\alpha}$-hard}\label{sec3.3}

This section contains the core of the reduction.  The goal of the
section is to establish a reduction from $\Sigma_{2\alpha}$-complete
sets to trees of rank $\leq \alpha$.  This is one quantifier too few,
but the result is also too strong -- the functions produced from the
resulting trees all reveal their rank at a uniform oscillation
sensitivity $\varepsilon = \frac{1}{4}$.  It will be a simple matter
later to encode another quantifier by producing functions whose ranks
are witnessed at non-uniform oscillation sensitivity.  (Indeed, this
is exactly the approach suggested by the definition of the rank.)

The next two lemmas are routine to establish, and their purpose is
to specify exactly how to strip two quantifiers off most
$\Pi_\alpha$ facts in a particularly nice way, a way which will be
useful for the main argument which is coming up in Lemma
\ref{technical}.

The first lemma takes an arbitrary $\Pi_{\alpha+2}$ fact and rewrites
it in a nice form, with unique witnesses and stable evidence.  In the
process, two computable reduction functions $g_0$ and $g_s$ are
defined which will be used in Lemma \ref{technical}.

\begin{lem}\label{g0}
For any $\Pi_{\alpha+2}$ predicate P(x), there is a $\Pi_{\alpha}$ predicate $R(x,z,y)$ such that
\begin{enumerate}
\item $P(x) \iff \forall z\exists y R(x, z, y)$
\item $R(x,z,y_1) \wedge R(x,z,y_2) \implies y_1=y_2$ ($R$ has unique witnesses)
\item For $z_1<z_2$, $\neg \exists y R(x,z_1,y) \implies \neg \exists y R(x, z_2, y)$ ($R$ has stable evidence)
\item $R(x,z,y) \implies z < y$
\end{enumerate}
\end{lem}
\begin{proof}
We may as well assume that $P(x)$ is ``$x \notin \emptyset_{(\alpha+2)}$''.  For the case $\alpha = 0$, we define $R$ using a computable, total $\{0,1\}$-valued function $g_0$, and set $R(x,z,y) \iff g_0(x,z,y) = 1$.

  Let $e$ be a $\Pi_2$ index for $\overline{\emptyset''}$,
  i.e. $\phi_e$ is total and \mbox{$x\notin \emptyset'' \iff
    \forall v \exists w [\phi_e(x,v,w)=1]$.}  Define $$ g_0(x,z,y) =
  \begin{cases} 1 & \text{ if $y>z$ and for all $v<z$ there
        is $w<y$ such that} \\ & \text{$\phi_e(x,v,w)=1$ and $y$ is
        least such that this is true} \\ 0 & \text{
        otherwise.}\end{cases}$$
One may check that four conditions on $R$ are satisfied.

For the case $\alpha>0$, we define $R$ using a computable, total
function $g_s$ and set $R(x,z,y) \iff g_s(x,z,y) \not\in
\emptyset_{(\alpha)}$.  The construction that defines $g_s$ uses
movable markers to build $\Pi_\alpha$ sets with at most one element.  At
any moment there is one particular element being held which is linked
to a potential least-witness, and this element will be held for as
long as that witness seems viable.

Let $e$ be a universal $\Pi_3$ index, i.e. $\phi_e^X$ is total for all
$X$ and $$x\notin X''' \iff \forall u \exists v \forall w
[\phi_e^X(x,u,v,w) = 1].$$ The intended oracle $X$ is an inverse jump
of $\emptyset_{(\alpha)}$, so that $X' = \emptyset_{(\alpha)}$ and
$X''' = \emptyset_{(\alpha+2)}$.  But the claims of the lemma also
hold for an arbitrary $X$ when we let $P(x)$ be $x\notin X'''$ and
$R(x,z,y)$ be $g_s(x,z,y)\notin X'$.

 Define
  $W_{g(x,z)}^X$ in stages according to the following dynamic
  process.  At stage $s=0$, let $W_{g(x,z),0}^X = \{n: n\leq z\}$, and
  let $t_1 = 0$.  For each $s>0$, let $y^0_s$ and $y^1_s$ be
  respectively the smallest and second smallest elements of
  $\overline{W_{g(x,z),s-1}^X}$.  Check whether $(\forall u <
  z)(\exists v < t_s) (\forall w < s) [\phi_e^X(x,u,v,w)=1]$.  If this
  is so, put $y^1_s$ into $W_{g(x,z),s}^X$, and set $t_{s+1}=t_s$.
  If this is not so, put $y^0_s$ into $W_{g(x,z),s}^X$, and set
  $t_{s+1} = t_s+1$.

  Then define $$W_{g_{s}(x,z,y)}^X = \begin{cases}\mathbb{N} &
      \text{ if } y\in W^X_{g(x,z)}  \\ \emptyset &\text{
        otherwise.}\end{cases}$$
This has the effect that $g_{s}(x,z,y)\in X' \leftrightarrow y\in
W^X_{g(x,z)}$.  

Now let us verify the claims of the lemma, in the more general case
where $P(x)$ is $x\notin X'''$ and $R(x,z,y)$ is $g_s(x,z,y)\notin
X'$.

First we address the second claim, that $R$ has unique witnesses.  For a given $x,z,X$, let us verify that there is at
most one $y$ such that $g_{s}(x,z,y)\notin X'$.  Suppose $y^0_s$ does
not stabilize in the construction above.  Then
$\overline{W_{g(x,z)}^X}$ does not have a smallest element, so it is
empty, so $W_{g(x,z)}^X = \mathbb{N}$.  On the other hand if $y^0_s$
stabilizes, then let $s_0$ be such that for all $s>s_0$,
$y^0_{s_0}=y^0_s$.  Then for all $s>s_0$, it must be that $y^1_s$ is
put into $W_{g(x,z),s}^X$, so
$\overline{W_{g(x,z)}^X}=\{y^0_{s_0}\}$.  Thus in either case,
$W_{g_{s}(x,z,y)}^X = \mathbb{N}$ for all but at most one $y$, so
$g_{s}(x,z,y)\in X'$ for all but at most one $y$.

For the first claim, suppose that $x\notin X'''$.  This is true if and
only if \mbox{$\forall u \exists v \forall w [\phi_e^X(x,u,v,w)=1]$} In that
case, for all $z$, in the construction of $W^X_{g(x,z)}$, we see that $t_s$
stabilizes, because there is a $t$ for which \mbox{$(\forall u < z)(\exists
v< t)(\forall w)[\phi_e^X(x,u,v,w)=1]$.}  And conversely, if $t_s$
stabilizes for each $z$, then $x\notin X'''$.  We have $\lim_s t_s$
exists exactly when $\lim_s y_s^0$ exists, since they always change
together.  And $\lim y_s^0 = y$ exists exactly when
$\overline{W^X_{g(x,z)}} = \{y\}$, which is equivalent to saying
$g_s(x,z,y) \notin X'$.  Thus $x\notin X'''$ if and only if
$g_s(x,z,y) \notin X'$.

For the third claim, note that if $z_1<z_2$ then $\lim_s t_s(z_1) \leq \lim_s t_s(z_2)$ where $t_s(z)$ refers to the $t_s$-values associated to the construciton of $W^X_{g(x,z)}$.  Thus, if $W^X_{g(x,z_1)} = \emptyset$, then $W^X_{g(x,z_2)}=\emptyset$ as well.

Finally, for the last claim, if $y\in \overline{W^X_{g(x,z)}}$ then $y>z$ because
$\{n: n\leq z\}\subseteq W_{g(x,z)}^X$ from the outset.

\end{proof}

The next lemma
explicitly splits up the queries to a $\emptyset^{(\lambda)}$ oracle that
occur during the evaluation of a $\Pi_\lambda$ question.
The goal is to isolate the parts of the computation that can be done
using a weaker oracle.  In the proof we define a function $g_l$ which will be used in Lemma \ref{technical}.

\begin{lem}\label{gls}
Let $\lambda$ be a limit ordinal, given as a uniform supremum $\lambda = \sup_n \beta_n$.  For any $\Pi_\lambda$ predicate $P(x)$ there is a sequence of predicates $R_n$ such that $$P(x) \iff \bigwedge_n R_n(x)$$ where $R_n$ is $\Pi_{2\beta_n}$ for each $n$. Furthermore, the $R_n$ are uniformly computable from $P$ and $\lambda$.
\end{lem}
\begin{proof}
  We may assume that each $\beta_n$ is a successor ordinal, and that $P(x)$ is ``$x\notin \emptyset_{(\lambda)}$''.  Now we define $R_n$ by specifying a computable function $g_l$ below and letting $R_n(x) \iff g_l(x,\lambda,n) \notin \emptyset_{(2\beta_n)}$.

  Uniformly in any pair of constructive ordinals $\alpha<\beta$, there is a reduction from $\emptyset_{(\beta)}$ to $\emptyset_{(\alpha)}$.  (See for example \cite[Lemma 5.1]{ashknight}.)  And any standard encoding will have the property that $\langle z,
  n\rangle \geq n$.  Therefore, \mbox{$\emptyset^{(\lambda)}\!\uhr\! n$} is uniformly computable from $\lambda,n$ and
  $\emptyset_{(\beta_n)}$, in the sense that there is a partial recursive
  function $\sigma(\lambda,n,X)$ which halts and returns
  $\emptyset^{(\lambda)}\uhr n$ if 
  $X=\emptyset_{(\beta_n)}$.

Define $g(x,\lambda,n)$ by $$W^X_{g(x,\lambda,n)} = \begin{cases}
    \emptyset & \text{ if } \phi_{x,n}^{\sigma(\lambda,n,X)}(x)\uparrow \\
    \mathbb{N} & \text{ otherwise.}\end{cases}$$ Suppose that
$x\notin \emptyset_{(\lambda)}$.  This is true if and only
if $$\phi_x^{\emptyset^{(\lambda)}}(x)\uparrow \iff \forall n
\phi_{x,n}^{\emptyset^{(\lambda)}\uhr n}(x)\uparrow \iff \forall n
[g(x,\lambda,n)\notin \emptyset_{(\beta_n + 1)}].$$
Define $g_l(x,\lambda, n)$ so that $g_l(x,\lambda, n) \notin \emptyset_{(2\beta_n)} \iff g(x,\lambda,n) \notin \emptyset_{(\beta_n + 1)}$.  (Since $\beta_n$ is a successor ordinal, $\beta_n + 1 \leq 2\beta_n$.)
\end{proof}

The following lemma contains the heart of the reduction.  Given a
$\Pi_\alpha$ fact, we must build a tree of the appropriate limsup
rank.  Each node of this tree will be associated with a finite set of
$\Pi_\beta$ assertions for different ordinals $\beta$.  The behavior
of the subtree below a node is as follows.  If all the assertions are
true, then the rank of the subtree should be large, on the order of
the largest $\beta$ from the set of assertions.  But if some
$\Pi_\beta$ assertion is false, then the rank of the subtree should be
small, of a similar height to that $\beta$.  

The node achieves this behavior by selecting which assertions should
be given to each of its child-nodes.  The collection of $\Pi_\beta$
assertions, if all true, could be viewed as having a generalized
Skolem function which covers the first two quantifiers of every
assertion in the collection.  The previous two lemmas will ensure
that this Skolem function, if it exists, is unique.  The children try
to guess fragments of this unique Skolem function, and each child is
given a set of assertions which explore the fragment of the Skolem
function that the child provided.  The previous two lemmas will
ensure that if infinitely many children can correctly guess a fragment
of the generalized Skolem function, then (1) all the assertions of the
parent are true and (2) these children, having guessed all the right
witnesses, will achieve high rank.

On the other hand, if some assertion was false at the level of the
parent node, then since the guesses are only fragments, finitely many
children will still come up with lucky guesses which give them a pile
of true assertions, some of which could be very large compared with
the false assertion the parent had.  Therefore, the children also each
re-evaluate all of the non-maximal assertions from their parent node;
this damps the sup of the ranks of the children.

As for damping the limsup, cofinitely many children will automatically
dampen down their own ranks through exploring the false assertions
generated by their Skolem guesses, which were doomed guesses in a
situation where in fact no witnesses existed.  Thus the limsup of the
ranks of the children is damped.  There is a subtlety here.  If the
limsup is supposed to be damped below some limit ordinal, it is not
enough that each child get below that ordinal individually.  They have
to obey a common bound.  That is why, in step (\ref{recurse}) below,
when $\alpha_i$ is a limit ordinal, $M_i$ is chosen the way it is.

All of the complication that is to follow arises in order to deal with
the limit case.  When a node is given only one $\Pi_{\alpha+2}$
assertion, each of its children is simply given a single
$\Pi_{\alpha}$ assertion.  If $\alpha$ is finite, the resulting tree
has finite height and just one assertion per node.  On a first reading
it may be helpful to have this special case in mind.

Here is another example, this one for the simplest limit case.  If a
node is given a single $\Pi_{\omega}$ assertion, that assertion may
be broken up into assertions of size $\Pi_2,\Pi_4,\Pi_6,\Pi_8$, and so
on, such that the original assertion is true if and only if all the
sub-assertions are true.  In that case, most of the children of the
node end up totally empty, but of the ones that do not, the first one
evaluates only the $\Pi_2$  assertion, the second one evaluates the
$\Pi_2$ and $\Pi_4$ assertions, and so on.  If all the assertions are
true, then the childrens' ranks get bigger the more assertions they
evaluate, causing the rank of the whole tree to reach $\omega+1$.  But
if the $\Pi_{2n}$ assertion is false for some $n$, then every child
that evaluates that one has finite rank at most $n$, and
every child that does not evaluate that one has rank at most $n$ as well
(because it only evaluates small assertions).  So the tree as a
whole gets rank at most $n+1$.

\begin{lem}\label{technical}
Let $\alpha_1,\dots,\alpha_k>0$ be constructive ordinals, and let $x_1,\dots,x_k$ be any natural numbers.  Recursively in $\alpha_1,\dots,\alpha_k,x_1,\dots,x_k$, one may compute a well-founded tree $T$ such that
\begin{itemize}
\item $|T|_{ls} = \max_i \alpha_i + 1$ if $x_i \notin \emptyset_{(2\alpha_i)}$ for all $i$\\
\item $|T|_{ls} \leq \alpha_i$ whenever $x_i \in \emptyset_{(2\alpha_i)}$.
\end{itemize}
\end{lem}

\begin{proof}

In order to perform the induction we will actually prove something
slightly stronger.  If $x_i\in \emptyset_{(2\alpha_i)}$ for $\alpha_i$ a limit, given as $\alpha_i = \sup_n\beta_n$, then by Lemma \ref{gls} there is a least $z$ such that
$g_{l}(x_i,\alpha_i,z)\in \emptyset_{(2\beta_z)}$.  In this case, we
will ensure that $T$ also satisfies $
|T|_{ls}\leq
\beta_z + 1
$ for that least $z$.

Define $T$ recursively as follows.  Renumber the inputs so that $\alpha_1\geq\dots\geq\alpha_k$.  (Since all the ordinal notations are comparable, this step is computable).  The empty sequence is in $T$.  To
compute information about the $n$th child of the root, decode $n$ as
$n=\langle m_0,m_1,\dots,m_k\rangle$ and do the following:
\begin{enumerate}
\item Check that $m_0<m_1<\dots<m_k$.  If it is not, $T_n = \emptyset$.
\item For any $i$ such that $\alpha_i$ is a limit, check that \mbox{$m_i=m_{i-1}+1$.}  If it does not, then
  $T_n=\emptyset$.
\item\label{check} For any $i$ such that $\alpha_i=1$, check that $g_0(x_i,m_{i-1},m_i)=1$.  If it does not, then $T_n=\emptyset$.
\item\label{alpha1=1} If $\alpha_1=1$, $T_n=\{\emptyset\}$.
\item\label{recurse} Otherwise, we decide the subtree
  rooted at $\langle m_0,\dots,m_k\rangle$ according to membership in
  the tree which we will now specify.  Build a finite set
  $\mathcal{F}$ of ordinal-input pairs as follows.
\begin{itemize}
\item Let $\mathcal{F}_1 = \{(\alpha_i,x_i) : \alpha_i < \alpha_1\}$\\
\item Let $\mathcal{F}_2 = \{(\beta,g_s(x_i,m_{i-1},m_i)) : \alpha_i = \beta+1 \text{ where } \beta>0\}$\\
\item For each limit $\alpha_i = \sup_n \beta_n$, let $M_i\geq m_i$ be least such that for each $(\gamma,x)\in\mathcal{F}_1\cup\mathcal{F}_2$, if $\gamma<\alpha_i$, then $\gamma\leq  \beta_{M_i} $.  (Again, this $M_i$ may be effectively computed since the notations involved are all comparable.)  For each $n\leq M_i$, let $(\beta_n, g_l(x_i,\alpha_i,n))\in\mathcal{F}_3$.
\end{itemize}
Let $\mathcal{F} = \mathcal{F}_1\cup\mathcal{F}_2\cup\mathcal{F}_3$.  Then $T_n$ is defined recursively as the tree computed from the pairs in $\mathcal{F}$.
\end{enumerate}
 This completes the construction.

Observe that the resulting $T$ is well-founded because each time we
recurse, the size of the largest ordinal under consideration
decreases.  Let us verify the properties of this $T$. We proceed by
induction on the size of $\max_i \alpha_i$.

For now on, consider the $\alpha_i$ to be numbered in order, so $\max_i \alpha_i = \alpha_1$.

In the base case, $\alpha_1
= \dots = \alpha_k=1$. If $g_0(x_i,m_{i-1},m_i)=0$ for any $i$, then $T=\{\emptyset\}$ and
$|T|_{ls}=1$ which is correct.  If $g_0(x_i,m_{i-1},m_i)=1$ for all $i$, step (\ref{alpha1=1}) is encountered infinitely often and thus $|T|_{ls} = 2$, which is correct.

Now we consider the case $\alpha_1>1$. If, when computing subtree $T_n$, the algorithm makes it to step
\ref{recurse}, then we call $n$ a \emph{recursing child}.

By induction we may always assume that for each child of the root $n$,
\mbox{$|T_n|_{ls} \leq \alpha_1$}.  This follows because
$|T_n|_{ls}\leq 1$ for non-recursing children $n$, and for recursing
children $n$, the ordinals considered in order to decide subtree $T_n$
are all less than $\alpha_1$.  Therefore it is always true that
$|T|_{ls}\leq\alpha_1+1$.

\subsubsection*{Case 1: The rank should be large}

Suppose that for all $i$, $x_i\notin \emptyset_{(2\alpha_i)}$.  Let us see that in this case $|T|_{ls} =
\alpha_1+1$ is attained.  Recall that a child of the root $n$
is decoded as $n=\langle m_0,\dots,m_k\rangle$.  For each choice of
$m_0$, a certain child of the root is obtained by inductively choosing
$m_i$ as follows according to the nature of $\alpha_i$.  The functions $g_0$ and $g_s$ are as defined in Lemma \ref{g0}.
\begin{enumerate}
\item If $\alpha_i = 1$, choose $m_i$ so that $g_0(x_i,m_{i-1},m_i) = 1$,
\item If $\alpha_i = \beta+1$ with $\beta>0$, choose $m_i$ so that $g_{s}(x_i,m_{i-1},m_i)
  \notin \emptyset_{(2\beta)}$
\item If $\alpha_i$ is a limit, choose $m_i=m_{i-1}+1$.
\end{enumerate} Let $n_j$ be the child so constructed starting with $m_0=j$.
By the definitions of
$g_0$ and $g_{s}$, each $m_i$ described above exists, is unique, and
satisfies $m_i>m_{i-1}$.

One can check that $n_j$ is a recursing child, and so $T_{n_j}$ is
formed using a finite set of ordinal-index pairs $(\gamma,z)$.  Notice
that the choices of $m_i$ above, together with the fact that for all
$i$, $x_i\notin \emptyset_{(2\alpha_i)}$, guarantee
that $z\notin \emptyset_{(2\gamma)}$ for each of these
pairs $(\gamma,z)$. Therefore, $|T_{n_j}|_{ls}$ will be determined by the
largest ordinal under consideration in the construction of $T_{n_j}$.
Now if $a_1 = \beta+1$, then one of the pairs under consideration in
the construction of $T_{n_j}$ is $(\beta, g_{s}(x_1,m_0,m_1))$, and $\beta$ is maximal among ordinals
considered for $T_{n_j}$.  Therefore by the inductive hypothesis, for
each $j$ we have $|T_{n_j}|_{ls} = \beta + 1 =
\alpha_1$.  Since there are infinitely many child subtrees
where this rank is obtained, \mbox{$\limsup_n |T_n|_{ls} =
  \alpha_1$} and thus $|T|_{ls}= \alpha_1+ 1$ as
required.  On the other hand, if $\alpha_1=\sup_n\beta_n$ is a limit, then
$(\beta_{M_1},g_{l}(x_1,\alpha_1,M_1))$ is used when assembling $T_{n_j}$,
and $\beta_{M_1}$ is maximal among ordinals considered, because
if $\alpha_i<\alpha_1$, then $\beta_{M_1}\geq\alpha_i$, and if $\alpha_i = \alpha_1$, then $M_i = M_1$ (since their selection algorithms are identical).  Therefore, by the
inductive hypothesis, $$|T_{n_j}|_{ls} =
  \beta_{M_1}+1 >
  \beta_{j}+1$$ because $M_1\geq
m_1>m_0=j$.  Since $\lim_j \beta_j = \alpha_1$, we
have $$
\lim_j|T_{n_j}|_{ls} \geq \lim_j (\beta_j +
1) = \alpha_1$$ as well.  Therefore,
\mbox{$\limsup_n |T_n|_{ls} = \alpha_1$} and $|T|_{ls} =
\alpha_1+1$ as required.  Therefore, if for all $i$,
$x_i\notin \emptyset_{(2\alpha_i)}$, then
\mbox{$|T|_{ls} = \alpha_1+1$.}

\subsubsection*{Case 2: The rank should be small}

On the other hand, suppose that $x_i\in {\emptyset_{(2\alpha_i)}}$ for some $i$.  Fix an index $r$ at which this occurs.  We will
show that $|T|_{ls}\leq \alpha_r$.

\emph{Subcase 2.1} Suppose $\alpha_r=\beta_r+1$.  By Lemma \ref{g0} let
$z_r$ be such that $$
(\forall z>z_r)(\forall y>z) [g_{s}(x_r,z,y)\in
\emptyset_{(2\beta_r)}]
$$ if $\beta_r>0$, or such that $(\forall z> z_r)(\forall
y>z)[g_0(x_r,z,y)=0]$ if $\beta_r=0$.  One may check that for any
child $n=\langle m_0,\dots,m_k\rangle$ such that $m_{r-1}>z_r$, if $n$ is recursing, then
included in consideration for $T_n$ is $(\beta_r,g_{s}(x_r,m_{r-1},m_r))$
where $g_{s}(x_r,m_{r-1},m_r)\in \emptyset_{(2\beta_r)}$; and if $n$ is not recursing, $T_n = \emptyset$.  Therefore by induction,
\mbox{$|T_n|_{ls}\leq \beta_r < \alpha_r$} for such
$n$.

Now let us consider recursing children $n$ such that $m_{r-1}\leq
z_r$.  There are only finitely many ways $m_0<\dots<m_{r-1}\leq z_r$
to begin such children.  Fix one such beginning.  We claim that for
all but at most one choice of the remaining $m_r<\dots<m_k$,
$|T_n|_{ls}< \alpha_r$.  That one choice, if it exists, is constructed
inductively as in the previous case.  That is, for each $i\geq r$, choose $m_i$ to satisfy
\begin{enumerate}
\item If $\alpha_i = 1$, satisfy $g_0(x_i,m_{i-1},m_i) = 1$,
\item If $\alpha_i = \beta+1$
  with $\beta>0$, satisfy $g_{s}(x_i,m_{i-1},m_i) \notin \emptyset_{(2\beta)}$, and
\item If $\alpha_i$ is a limit, let $m_i = m_{i-1}+1$.
\end{enumerate}
If these $m_i$ exist, they are unique.  Suppose we deviate from this
recipe in the case of $\alpha_i$ a limit.  Then $T_n$
is empty.  Suppose we deviate from this one way in the case of $\alpha_i =
1$, and let $g_0(x_i,m_{i-1},m_i)=0$.  Then by step (\ref{check}), $T_n$ is empty.  Suppose we deviate from
this one way in the case of $\alpha_i = \beta+1$, and include $(\beta,
g_{s}(x_i,m_{i-1},m_i))$ in the assembling of $T_n$, where
$g_{s}(x_i,m_{i-1},m_i)\in \emptyset_{(2\beta)}$.  Then by the inductive
hypothesis we are guaranteed $|T_n|_{ls}\leq \beta <
\alpha_i \leq \alpha_r$. Therefore,
considering all children $n$, there are at most finitely many such that
$|T_n|_{ls}\geq \alpha_r$.  Therefore, $\limsup_n
|T_n|_{ls} \leq \beta_r$.

It remains to show that for each recursing child $n$, $|T_n|_{ls} \leq
\alpha_r$.  There are two possibilities.  If $\alpha_1>
\alpha_r$, then $(\alpha_r,x_r)$ is included in consideration for $T_n$, and
thus by the inductive hypothesis $|T_n|_{ls} \leq \alpha_r$.  On the other hand, if $\alpha_1=\alpha_r=\beta_r+1$, then $\beta_r$ is maximal among
ordinals considered for $T_n$, so by the inductive hypothesis
$|T_n|_{ls} \leq \beta_r + 1=\alpha_r$.
Therefore, if $\alpha_r$ is a successor, then $|T|_{ls}\leq
\alpha_r$.

\emph{Subcase 2.2}: Suppose $\alpha_r =\sup_n\beta_n$ is a limit.  Using Lemma \ref{gls}, let $z_r$ be least such that \mbox{$g_{ls}(x_r,\alpha_r,z_r) \in
\emptyset_{(2\beta_{z_r})}$.}  Let us consider children $n=\langle
m_0,\dots,m_k\rangle$ such that $m_r\geq z_r$.  For each of these $n$,
the pair $(\beta_{z_r},g_{l}(x_r,\alpha_r,z_r))$ is used in assembling $T_n$.  So
for each such $n$, $|T_n|_{ls} \leq \beta_{z_r}$.

On the other hand, there are the $n$ such that $m_r<z_r$.  There are
only finitely many ways $m_0<\dots<m_r<z_r$ to begin such an $n$.  We
claim that for each such beginning, there is at most one sequence
$m_{r+1},\dots,m_k$ which completes $n$ in such a way that
$|T_n|_{ls}>\beta_{z_r}$.  The strategy is exactly the
same as in the successor case.  See (1)-(3) above.

In each case, if such an $m_i$ exists, it is unique.  If we deviate
from this plan in the case of $\alpha_i$ a limit or $\alpha_i=1$, then
one may check that $T_n$ is empty.  If we deviate in the case of
$\alpha_i=\beta+1$ with $\beta>0$, then we include
$(\beta,g_{s}(x_i,m_{i-1},m_i))\in\mathcal{F}_2$, where
$g_{s}(x_i,m_{i-1},m_i)\in \emptyset_{(2\beta)}$.  So to start with,
$|T_n|_{ls}\leq \beta$, and if $\beta\leq \beta_{z_r}$
 then $|T_n|_{ls}$ is small enough.  But if
$\beta> \beta_{z_r}$, then this bound is insufficient.  In that
case, recall that during the creation of $\mathcal{F}_3$ which was
used to assemble $T_n$, we defined $M_r$ to satisfy $M_r\geq m_r$ and
$\beta_{M_r}\geq \gamma$ for each
$(\gamma,z)\in\mathcal{F}_{2}$ such that $\gamma < \alpha_r$.  Because $\alpha_i<\alpha_r$, $\beta < \alpha_r$.  So
$\beta_{M_r}\geq \beta > \beta_{z_r}$, so $M_r>z_r$.  So in particular,
$(\beta_{z_r},g_{l}(x_r,\alpha_r,z_r))$ was included when assembling
$T_n$.  Therefore, $|T_n|_{ls}\leq \beta_{z_r}$.
Therefore, for all but finitely many $n$,
$|T_n|_{ls}\leq \beta_{z_r}$.

It remains to show that for each individual $n$, $|T_n|_{ls}\leq
\beta_{z_r}+ 1$.

We now consider two cases.  Suppose $\alpha_1 > \alpha_r$.  Then for each
$n$, the pair $(\alpha_r,x_r)$ is again under consideration for $T_n$.  But
the new leading ordinal is smaller, so by induction, $|T_n|_{ls} \leq
\beta_{z_r} + 1$ for each $n$.  On the other hand,
if $\alpha_1 = \alpha_r$, then $\alpha_1=\alpha_2=\dots=\alpha_r$, so
$M_1=M_2=\dots=M_r$, since the algorithm which selects $M_i$ is the same for each $i = 1,\dots, r$. One may check that $\beta_{M_r}$ is the largest ordinal
under consideration in the assembling of $T_n$.  If
$\beta_{M_r} \geq \beta_{z_r}$, then $(\beta_{z_r},g_{l}(x_r,\alpha_r,z_r))$
is included, and $|T_n|_{ls}\leq \beta_{z_r}$.  On the other hand, if
$\beta_{M_r} < \beta_{z_r}$, then since $\beta_{M_r}$
is largest, $$
|T_n|_{ls}\leq \beta_{M_r} + 1 \leq \beta_{z_r}.
$$
Therefore, $\sup_n |T_n|_{ls} \leq \beta_{z_r} + 1$.
Therefore, if $\alpha_r$ is a limit with $g_{l}(x_r,\alpha_r,z_r)\in \emptyset_{(2\beta_{z_r})}$, then  $|T|_{ls}\leq \beta_{z_r} + 1 < \alpha_r$.  This
completes the proof.
\end{proof}

\subsection{Recognizing functions of rank $\alpha$ is $\Pi_{2\alpha+1}$-hard}\label{sec3.4}

In this section we obtain the final result by consideration of what
can be encoded into the oscillation sensitivity $\varepsilon$ at which
a function's rank is witnessed.  For this last step, it is necessary
to consider functions again instead of trees, because with the trees
we only produce functions made of bumps with all the same proportions.
In the next theorem, we use functions made of increasingly shallow
bumps, and encode the last jump into the uncertainty of how small
$\varepsilon$ will have to be in order for the bumps which determine
the function's rank to be detectable.

\begin{thm}\label{noncompact-trick}
Uniformly in a constructive ordinal $\alpha>0$ and $x$, one may find a computable $f\in C[0,1]$ satisfying
\begin{itemize}
\item $x\notin \emptyset_{(2\alpha + 1)} \rightarrow |f|_{KW} \leq \alpha$
\item $x\in \emptyset_{(2\alpha + 1)} \rightarrow |f|_{KW} = \alpha + 1$
\end{itemize}
\end{thm}

\begin{proof}
Given $\alpha,x$, compute $f$ as follows.  Similar to earlier, let $\{[a_n,b_n]\}_{n\in\mathbb{N}}$ be any computable
  sequence of intervals with rational endpoints satisfying
\begin{itemize}
\item Each interval is contained in $(0,1)$
\item $b_{n+1}<a_n<b_n$ for each n.
\item $\lim_{n\rightarrow\infty}a_n = 0$
\item $b_n-a_n < a_n^2$
\end{itemize} 
Let $g$ be a computable function satisfying for all $x$ and $X$, 
$$
x\in X'' \iff \exists s [g(x,s)\notin X'].
$$ 
Then  $$
x\in \emptyset_{(2\alpha+1)} \iff \exists s [g(x,s)\notin \emptyset_{(2\alpha)}].
$$ 
For any $s$, let $T(s)$ be the tree
guaranteed by Lemma \ref{technical} with input $(\alpha, g(x,s))$.  Thus $|T(s)|_{ls} = \alpha + 1$ if $g(x,s)\notin \emptyset_{(2\alpha)}$ and $|T(s)|_{ls} \leq \alpha$ otherwise.  Then define 
$$ 
f =\sum_{s=0}^\infty \frac{1}{s+1}f_{T(s)}[a_s,b_s].
$$ 
Recall that Proposition \ref{goodfunction} guarantees that
$||f_{T(s)}||<1$,
so $$||\frac{1}{s+1}f_{T(s)}[a_s,b_s]|| < \frac{b_s-a_s}{s+1} <
\frac{a_s^2}{s+1}.$$ On neighborhoods bounded away from 0,
$f$ is a uniformly presented sum of finitely many computable
differentiable functions, but $f$ lives in the envelope of
$x^2$, so it is computable near $0$ as well.  Thus $f$ is
computable and differentiable.

Suppose $x\notin \emptyset_{(2\alpha+1)}$.  Then for each $s$, $g(x,s)\in \emptyset_{(2\alpha)}$, so
\mbox{$|T(s)|_{ls} \leq \alpha$,} so
$|f_{T(s)}|_{KW} \leq \alpha$.  For each $z\neq 0$,
there is a neighborhood $U$ of $z$ which intersects exactly one of the
$[a_s,b_s]$.  Because
\mbox{$P^\alpha_{f_{T(s)}[a_s,b_s],\varepsilon} =
  \emptyset$} for all $\varepsilon$, and $f_{T(s)}[a_s,b_s]$
coincides with $f$ on $U$, Lemma \ref{A} implies that $z
\notin P_{f,\varepsilon}^\alpha$ for any
$\varepsilon$.  On the other hand, fix $\varepsilon$ and let $z=0$.
Then for any $s$, by Proposition \ref{goodfunction}, $||f_{T(s)}'||<2$, so
$$||\frac{1}{s+1}f_{T(s)}'[a_s,b_s]||= \frac{1}{s+1}||f_{T(s)}'||<\frac{2}{s+1}.$$  Let $S$
be large enough that $\frac{4}{S+1}<\varepsilon$.  Then for all
$p,q,r,s \in [0,b_S)$,
\begin{align*}|\Delta_{f}(p,q)-\Delta_{f}(r,s)|&\leq
|\Delta_{f}(p,q)|+|\Delta_{f}(r,s)| \\ &\leq
2||f'\uhr[0,b_S)|| < \frac{4}{S+1} < \varepsilon,\end{align*} so
$0\notin P_{f,\varepsilon}^\beta$ for any $\beta>0$.
Therefore $P_{f,\varepsilon}^{\alpha} = \emptyset$ for all $\varepsilon$
and $|f|_{KW} \leq \alpha$.

On the other hand, suppose that $x\in \emptyset_{(2\alpha+1)}$.  Let $s$ be such
that \mbox{$g(x,s)\notin \emptyset_{(2\alpha)}.$} Then $T(s)$ has
rank $\alpha+1$.  So $|f_{T(s)}|_{KW} = \alpha+1$, and this rank is visible at oscillation sensitivity
 $\varepsilon=\frac{1}{4}$ by Proposition \ref{tree-language}.
So also $|\frac{1}{s+1}f_{T(s)}|_{KW} = |\xi^{-1}(a)|_\kO+1$, and this rank is visible at oscillation sensitivity $\varepsilon = \frac{1}{4(s+1)}$.  Therefore by Lemmas \ref{B} and \ref{A},
$$\emptyset \neq P^{\alpha}_{\frac{1}{s+1}f_{T(s)}[a_s,b_s],\frac{1}{4(s+1)}} \subseteq P^{\alpha}_{f,\frac{1}{4(s+1)}}.$$  Thus $|f|_{KW} \geq \alpha + 1$.  Also, for each $s$, 
\mbox{$|f_{T(s)}|_{KW}\leq\alpha+1$,} and $0\notin P^\beta_{f,
\varepsilon}$ for any $\varepsilon$ and any $\beta>0$, so just as above, 
\mbox{$|f|_{KW}\leq \alpha +1$} always.  So in fact $|f|_{KW}=\alpha+1$.
\end{proof}

Therefore, we have the following:
\begin{thm}\label{maintheorem}
  For each nonzero $\alpha<\omega_1^{CK}$, $D_{\alpha+1}$ is
  $\Pi_{2\alpha+1}$-complete.
\end{thm}
\begin{proof}  By Proposition \ref{prop2}, $D_{\alpha+1} \leq_m
  \overline{\emptyset_{(2\alpha+1)}}$.  By Theorem \ref{noncompact-trick},
  $\overline{\emptyset_{(2\alpha+1)}}\leq_m D_{\alpha+1}$.
\end{proof}

\begin{thm}\label{limittheorem}
For any limit ordinal $\lambda < \omega_1^{CK}$, $D_\lambda$ is
$\Sigma_{\lambda}$-complete.
\end{thm}
\begin{proof}
First we show that $D_\lambda$ is $\Sigma_{\lambda}$.  Given $\lambda = \sup_n\beta_n$, we have $e\in D_\lambda \iff \exists n[e\in D_{\beta_n+1}]$.  Each $e\in D_{\beta_n+1}$ is $\Pi_{2\beta_n+1}$ by Proposition \ref{prop2},
so $D_\lambda$ is $\Sigma_{\lambda}$.

Now we show that $D_\lambda$ is $\Sigma_{\lambda}$-complete by
giving an appropriate reduction.  
We claim that $$x\in \emptyset_{(\lambda)} \iff \left| f_{T}\right|_{KW} < \lambda,$$ where $T$
is the tree constructed in Lemma \ref{technical} from input $(\lambda, x)$.  That lemma
guarantees first that $x\notin \emptyset_{(\lambda)}$ implies $|T|_{ls} = \lambda + 1$.
Conversely, if $x\in \emptyset_{(\lambda)}$ we have $|T|_{ls}
\leq \lambda$.  But by Proposition \ref{tree-language}, the limsup
rank of a tree is always a successor, so in fact $x \in \emptyset_{(\lambda)}$
implies $|T|_{ls} <\lambda$.
Thus $x\in \emptyset_{(\lambda)} \iff |T|_{ls}
  <\lambda \iff \left|f_T\right|_{KW}
  <\lambda.$
\end{proof}

\bibliographystyle{amsalpha}
\bibliography{ddbib}{}

\end{document}